\documentclass[11pt, a4paper]{article} %tamaño mínimo de letra 11pto.

\usepackage[english]{babel} %English 
\usepackage[utf8]{inputenc} %Para poder poner tildes
\usepackage{vmargin} %Para modificar los márgenes
\setmargins{2.5cm}       % margen izquierdo
{1.5cm}                        % margen superior
{16.5cm}                      % anchura del texto
{23.42cm}                    % altura del texto
{10pt}                           % altura de los encabezados
{1cm}                           % espacio entre el texto y los encabezados
{0pt}                             % altura del pie de página
{2cm}                           % espacio entre el texto y el pie de página

\usepackage{amsmath}
\usepackage{amssymb}
\usepackage{amsthm}

\usepackage[final]{graphicx}
\usepackage{soul}
\usepackage{comment}
\usepackage{vmargin}
\usepackage{braket}
\usepackage{enumerate}
\usepackage{enumitem} % allows customizing enumerate environment
\setlist[enumerate,1]  {label={\rm (\roman*)}, leftmargin=1.5em}   %{label={(\arabic*)}} 
\setlist{noitemsep} % no spacing between items

\usepackage{empheq} 
\include{extsizes}
\usepackage{bigstrut}
\usepackage{float} 
\usepackage{mathrsfs} % extra nice fonts
\usepackage{mdframed} % para los entornos con un “borde"
\usepackage{xcolor} % para los colores

\usepackage[nottoc]{tocbibind}
\usepackage{amsthm}
\usepackage{wrapfig}
\usepackage{caption2}
\usepackage{esint}
\usepackage[title]{appendix} % Adds "appendix" to the appendixes
\usepackage{titlesec}
\usepackage{commath}
\usepackage{accents}

\numberwithin{equation}{section}
\newtheorem{theorem}{Theorem}[section]
\newtheorem{proposition}[theorem]{Proposition}
\newtheorem{lemma}[theorem]{Lemma}
\newtheorem{corollary}[theorem]{Corollary}
\newtheorem{definition}[theorem]{Definition}

\newtheorem{remark}[theorem]{Remark}

\renewenvironment{proof}[1][Proof] {\noindent \textbf{#1.} }
{\  \rule{0.5em}{0.5em}\par \medskip}

\newcommand{\hole}{\mathscr{C}}

\newcommand{\wto}{\rightharpoonup}
\newcommand{\myleq}[1]{\ensuremath{\stackrel{\text{#1}}{\leq}}}
\newcommand{\mygeq}[1]{\ensuremath{\stackrel{\text{#1}}{\geq}}}
\newcommand{\mywto}[1]{\ensuremath{\stackrel{\text{#1}}{\wto}}}
\newcommand{\defeq}{\vcentcolon=}

\newcommand\restr[2]{\ensuremath{{#1}_{|_{#2}}}}

\renewcommand{\abs}[1]{\left| #1 \right|}
\newcommand{\D}{\displaystyle}

\newcommand*\Lap{\mathop{}\!\mathbin\Delta}

\titleformat{\paragraph}
{\normalfont\normalsize\bfseries}{\theparagraph}{1em}{}
\titlespacing*{\paragraph}
{0pt}{3.25ex plus 1ex minus .2ex}{1.5ex plus .2ex}

\newcommand*\adh[1]{\overline{#1}} % (or \bar{#1})

\newcommand{\RN}{{\mathbb{R}^N}}

\newcommand{\R}{\mathbb{R}}

\newcommand{\bnorm}[1]{\mathrel{\left\|#1\right\|}}
\renewcommand{\norm}[1]{\mathrel{\|#1\|}}

\def\XXint#1#2#3{{\setbox0=\hbox{$#1{#2#3}{\int}$ }
		\vcenter{\hbox{$#2#3$ }}\kern-.580\wd0}}

 \usepackage[pagebackref]{hyperref}
\hypersetup{
  colorlinks=true,
  citecolor = blue, %%cyan,
  linkcolor= blue,
}%%%%%%%%%%%%
\usepackage{mathtools}
\mathtoolsset{showonlyrefs=true}

\begin{document}
	
	\title{Asymptotic behaviour of the heat equation in an exterior domain
		with general boundary conditions II. The case of
                bounded and of $L^{p}$ data.}

	\author{ Joaquín Domínguez-de-Tena${}^{*,1}$ \\
		An\'{\i}bal Rodr\'{\i}guez-Bernal\thanks{Partially supported
			by Projects PID2019-103860GB-I00 and  PID2022-137074NB-I00,  MICINN and   GR58/08
			Grupo 920894, UCM, Spain}\ ${}^{,2}$}

	\date{\today}
	\maketitle

	\setcounter{footnote}{2}
	\begin{center}
		Departamento de Análisis Matemático y  Matem\'atica Aplicada\\ Universidad
		Complutense de Madrid\\ 28040 Madrid, Spain \\ and \\
		Instituto de Ciencias Matem\'aticas \\
		CSIC-UAM-UC3M-UCM\footnote{Partially supported by ICMAT Severo Ochoa
			Grant CEX2019-000904-S funded by MCIN/AEI/ 10.13039/501100011033} , Spain 
	\end{center}
	
	\makeatletter
	\begin{center}
		${}^{1}${E-mail:
			joadomin@ucm.es}
		\\ 
		${}^{2}${E-mail:
			arober@ucm.es}
	\end{center}
	\makeatother

	\noindent {$\phantom{ai}$ {\bf Key words and phrases:}  Heat equation,
		exterior domain, asymptotic behaviour, decay rates, bounded initial data, Dirichlet, Neumann,
		Robin boundary conditions, bounded initial data.} 
	\newline{$\phantom{ai}$ {\bf Mathematical Subject Classification
			2020:} \
		35K05, 35B40, 35B30, 35E15}

	% 	\normalsize
	% 	\newpage
	% \tableofcontents
	\begin{abstract}
            In this work, we study the asymptotic behaviour of
            solutions to the heat equation in exterior domains, i.e.,
            domains which are the complement of a smooth compact set
            in $\RN$. Different homogeneous boundary conditions are
            considered, including Dirichlet, Robin, and Neumann
            ones. In this second part of our work, we   consider the
            case of bounded initial data and prove that, after some
            correction term, the solutions become close to the
            solutions in the whole space and show how
            complex behaviours appear. We also analyse the case of
            initial data in $L^p$ with $1<p<\infty$ where all
            solutions essentially decay to $0$ and the convergence
            rate could be arbitrarily slow. 
	\end{abstract}

	\normalsize

%\newpage
%\tableofcontents
%\newpage
\section{Introduction}

In this paper, which is a continuation of \cite{DdTRB24a} and \cite{DdTRB23},   we consider the heat equation
  \begin{equation} \label{eq:intro_heat}
	\left\{
	\begin{aligned}
		u_t-\Lap u = 0 \quad & in \ \Omega\times(0,\infty) \\
		B(u)=0 \quad & on \ \partial\Omega\times (0,\infty) \\
		u=u_0 \quad & in \ \Omega\times\{0\} , 
	\end{aligned}	
	\right. 
\end{equation}
in   a connected unbounded exterior domain $\Omega$, that is, the
complement of a compact set $\hole$ that we denote the \emph{hole},
which is the closure of a bounded smooth set; hence,
$\Omega=\RN\backslash \hole$.  We will assume,  without loss of
generality, that  $0\in \mathring{\hole}$,
the interior of the hole, and observe that $\hole$ may have different
connected components, although $\Omega$ is connected. The boundary
conditions, to be made precise in Section \ref{sec:preliminares}, include
Dirichlet, Neumann and Robin ones,  of the form $B(u)= \frac{\partial u}{\partial 
  n}+b u=0$ with $b >0$.

Our goal is to analyse the effect on the solutions due to the
presence of the hole and the boundary conditions.

For the case of
integrable data in the whole domain (that is, no hole present)  it is shown in
\cite{duoandikoetxea, giga, vazquez2017asymptotic} that the mass of
the solution, that is the integral in $\R^{N}$ of the solution, which
remains invariant in time  is distributing in space as a multiple of
the Gaussian $G(x,t)
=\frac{e^{-\frac{\abs{x}^2}{4t}}}{(4\pi t)^{\frac{N}{2}}}$.  This is a
phenomenon of the \emph{mass moving to infinity} with time. 
On the other end, for bounded data in $\R^{N}$ it has been shown in 
\cite{vazquez2002complexity,cazenave2003universal,
  wang2012complicated} that bounded solutions of the heat equation
show  complex dynamical  behaviour. See also  \cite{robinson2018optimal} for the
case of unbounded initial data and the phenomenon of \emph{mass moving from
infinity}. See also
\cite{vazquez2002complexity, cazenave2003universalb,
  wang2018complicated, robinson2023estimates}
for other related equations.

In the case of domain with holes, the holes  and boundary conditions 
introduce some dissipative mechanism in the equation so solutions
behave different from the case in the whole space. The case of
integrable data has been analysed in our previous works
\cite{DdTRB23,DdTRB24a} and also in
\cite{herraiz1998,quiroscanizo24}. It is shown
in these references that, contrary what it happens in the whole space, 
some computable fraction of the  mass of the initial data is lost through
the hole and then the  remaining mass distributes in space, asymptotically
in time, as the  Gaussian function $G(x,t)$ times a
correction profile function $\Phi(x)$ 
that takes into account the boundary conditions and satisfies $\Phi(x) \to 1$ as $|x|\to
\infty$.  In particular,  far from the hole solutions behave as those of the
problem in $\R^{N}$. 
Related results for porous medium equations can be found in 
\cite{quiros1999asymptotic, quiros2007, quirospm2d} and for some
nonlocal problems in \cite{cortazar2012asymptotic,CORTAZAR2016586}.

In this paper we analyse the case of bounded initial data and $L^{p}(\Omega)$
initial data with $1<p< \infty$. In the former  case of bounded data and for
$N\geq 3$,  although
no mass is associated to initial data we show that still the Gaussian
and the profile $\Phi$ describe the behaviour of solutions and show
complex asymptotic behaviour of solutions. In the
latter case we show that all solutions decay to zero as the equation
is somehow more dissipative than the case of integrable or bounded
initial data.

The organization of the paper follows. In Section \ref{sec:preliminares} we introduce some
previous results for the solutions of \eqref{eq:intro_heat} for $L^p(\Omega)$ and bounded
initial data. As the main results on bounded solutions rely on
comparison with suitable super and subsolutions, in Section
\ref{sec:bounded-very-weak_solutions} we develop those technical
results for very weak  bounded solutions.  
Then in  Section \ref{sec:initial_data_Linfty} for  initial data 
$u_{0}\in L^{\infty}(\Omega)$, we first prove that, for any
boundary conditions, far from the hole, the solutions of
\eqref{eq:intro_heat} remain close to the solutions of the heat 
equation in $\R^{N}$ with the same initial data (extended by zero out
of $\Omega$), $u_\RN(x,t)$, see Theorem
\ref{thr:dataLinfty_far_from_hole}. Then, we prove in Theorem
\ref{thm:dataLinfty_Dirichlet} that if 
$N\geq 3$ and  $u_0\in L^\infty(\Omega)$, then
\begin{equation}
\label{eq:intro_convergence_Linfty}
  \lim_{t\to\infty}\norm{u(t) -\Phi (\cdot)u_\RN(t)}_{L^\infty(\Omega)}=0 
\end{equation}
and more precisely 
\begin{equation}
  \label{eq:intro_rate_of_convergence_Linfty}
  \norm{u(t)-\Phi(\cdot)u_\RN(t)}_{L^\infty(\Omega)} \leq \left\{
\begin{aligned}
& \frac{C\log(t)}{\sqrt{t}}\norm{u_0}_{L^\infty(\Omega)} &&
\text{if}  \ N=3 \\
& \frac{C}{\sqrt{t}}\norm{u_0}_{L^\infty(\Omega)}  && \text{if}  \ N>3
\end{aligned}
\right.
\end{equation}
for large $t>0$.
This allows us to prove that
   \eqref{eq:intro_heat} inherits some of the complex behavior of the
solutions of the heat equation in $\R^{N}$ with bounded data in
\cite{vazquez2002complexity}, see Theorems \ref{thm:compleq1} and
\ref{thm:dowhatyouwant}.  

In the two dimensional case, $N=2$, and except for Neuman boundary
conditions, the profile $\Phi$, that is constructed as in
\eqref{eq:asymptotic_profile_parabolic}, vanishes in $\Omega$. This
implies that for any bounded initial data, the solution $u(x,t)$
converges to zero as $t\to \infty$ uniformly in compact
sets. Therefore \eqref{eq:intro_convergence_Linfty} and
\eqref{eq:intro_rate_of_convergence_Linfty} are no longer true.
Hence the two
dimensional case requieres  a separate analysis that will be carried
out elsewhere.

Finally, in Section \ref{sec:initial_data_Lp} we consider initial data 
$u_{0}\in L^{p}(\Omega)$ for $1<p<\infty$. Unlike the cases $p=1$ and
$p=\infty$, where the correspondent norm of the
solutions does not decay and solutions approach some asymptotic
profile, when $1<p<\infty$ we show that all solutions
converge 
to zero in  $L^{p}(\Omega)$, as the equation is somehow more dissipative in these
spaces. More precisely, we prove in Proposition \ref{prop:lpasym} that, for any $1<p<\infty$ and  $u_0\in L^p(\Omega)$ 
\begin{equation}
\label{eq:decay_Lp}
  \lim_{t\to\infty}\norm{u(t)}_{L^p(\Omega)}=0 
\end{equation}	
and for any $q$ such that $p< q \leq \infty$,
\begin{equation}
	\lim_{t\to\infty}t^{\frac{N}{2}(\frac{1}{p}-\frac{1}{q})} \norm{u(t)}_{L^q(\Omega)}=0.
\end{equation}
Also, there is no  decay rate in \eqref{eq:decay_Lp} as  for any
decay function, there exist solutions that decay slower, see Lemma 
\ref{lemma:souplet}. 
In addition, we included two appendixes were we collected
some auxiliar technical  results needed for some of the proofs.

Throughout  this paper, we adopt
the convention of using $c$ and  $C$ to represent various constants 
which may change from line to line, and whose concrete value is not
relevant for the results.

\section{Notations and some  preliminary  results}
\label{sec:preliminares}

All along this paper we  consider an exterior domain  $\Omega =
\RN\backslash \hole$ as described in the Introduction, that is, the
complement of a compact set $\hole$, the \emph{hole},
which is the closure of a bounded smooth set and we will assume 
$\partial \Omega$ is of class $C^{2,\alpha}$ for some $0<\alpha<1$. 
We will also assume,  without loss of generality,  $0\in \mathring{\hole}$,
the interior of the hole, and observe that $\hole$ may have different
connected components, although $\Omega$ is connected. In this section
we use the notations and settings in \cite{DdTRB23,DdTRB24a} which we present below. 

In $\Omega$ we consider  the heat equation 
\begin{equation} \label{eq:heat_theta} 
	\left\{
	\begin{aligned}
		u_t-\Lap u = 0 \quad & in \ \Omega\times(0,\infty) \\
		B_\theta(u)=0 \quad & on \ \partial\Omega\times[0,\infty) \\
		u=u_0 \quad & in \ \Omega\times\{0\} , 
	\end{aligned}	
	\right. 
\end{equation}
where   we  consider Dirichlet, Robin or
Neumann homogeneous boundary conditions on $\partial \Omega$,  written in the form 
\begin{equation}
	\label{eqn:thetabc}
	B_\theta(u)\defeq \sin(\frac{\pi}{2}\theta(x))\frac{\partial u}{\partial n}+\cos(\frac{\pi}{2}\theta(x))u,
\end{equation}
where $\theta: \partial \Omega \longrightarrow [0,1]$ is of class
$C^{1,\alpha}(\partial \Omega)$ for some $0<\alpha<1$ and 
satisfies either  one of the
following cases in each connected component of $\partial \Omega$: 
\begin{enumerate}
	\item Dirichlet conditions: $\theta\equiv 0$
	\item Mixed Neumann and Robin conditions:
	$0<\theta\leq 1$. 
\end{enumerate}
In particular, if $\theta\equiv 1$ we recover Neumann
boundary conditions. 
In general, we will refer to these as homogeneous $\theta$-boundary
conditions.    Note that, by suitably choosing $\theta(x)$,
\eqref{eqn:thetabc} includes all boundary conditions of the 
form $\frac{\partial u}{\partial n}+b(x)u=0$. The restriction $0\leq
\theta \leq 1$ makes $b(x)\geq 0$ which is the standard dissipative
condition. The reason for these notations will be seen in the results
below about monotonicity of solutions with respect to $\theta$, see
\eqref{eqn:neugeqdir2} and \eqref{eq:comparison_kernels_theta}.

	As a general notation, for a given function $\theta$ as above, we  define the Dirichlet part of $\partial \Omega$ as
	\begin{displaymath}
		\partial^D \Omega\defeq \{x \in \partial \Omega \ : \ \theta(x)=0\},
	\end{displaymath}
	the Robin part of $\partial \Omega$ as
	\begin{displaymath}
		\partial^R \Omega\defeq \{x \in \partial \Omega \ : \ 0<\theta(x)<1\},
	\end{displaymath}
	and  the Neumann part of $\partial \Omega$ as
	\begin{displaymath}
		\partial^N \Omega\defeq \{x \in \partial \Omega \ : \
		\theta(x)=1\} . 
	\end{displaymath}	
	The conditions imposed on $\theta$ imply that $\partial^D \Omega$ is a
	union of connected components of $\partial \Omega$, although Neumann
	and Robin conditions can coexist in the same connected component of
	$\partial \Omega$.

  In general we will use a superscript $\theta$ to denote anything
  related to \eqref{eq:heat_theta}.  For example, the semigroup of
  solutions to \eqref{eq:heat_theta} will be denoted by $S^\theta(t)$
  and the associated kernel by $k^\theta(x,y,t)$. Sometimes, we may add as subscript
  $\Omega$ to indicate the dependence of these objects on the domain.

Then, we recall some results from  \cite{DdTRB23,DdTRB24a}. First,
\eqref{eq:heat_theta} defines a semigroup of solutions as 
$u(t;u_{0})= S^{\theta}(t) u_{0}$  for several classes of initial
data.  Actually the  semigroup $\{S^\theta(t)\}_{t>0}$   is  an order
preserving  semigroup of contractions in $L^p(\Omega)$ 
for $1\leq p\leq\infty$ which is $C^0$ if $p\neq \infty$ and
analytic if $1<p<\infty$. It is also   a  strongly continuous  analytic
semigroup in
$BUC_\theta(\Omega)$ which is the subspace of the space of bounded and
uniformly continuous functions in $\Omega$,  $BUC(\Omega)$, that vanish in the connected components of $\partial
\Omega$ in which $\theta=0$, see \cite{mora1983semilinear}. In particular, for  any
$u_0$ in such classes, 
	\begin{displaymath}
		\abs{S^\theta(t)u_0(x)}\leq S^\theta(t)\abs{u_0}(x), \qquad x\in
		\Omega, \ t>0 . 
	\end{displaymath}

Also, 
for $1\leq p\leq \infty$, 
\begin{equation}
	\label{eqn:Sexchange}
	\int_\Omega fS^\theta(t)g = \int_\Omega gS^\theta(t)f  \qquad
	\mbox{for all} \  f\in L^p(\Omega), \   g\in L^q(\Omega) 
\end{equation}
where $q$ is the conjugate of $p$,  that is
$\frac{1}{p}+\frac{1}{q}=1$. Hence, for $1\leq p<\infty$, the semigroup in $L^{q}(\Omega)$
is the adjoint of the semigroup in $L^{p}(\Omega)$. In particular, the
semigroup in $L^{\infty}(\Omega)$ is weak-* continuous.

Observe that in  $L^p(\Omega)$ for  $1<p<\infty$, the generator of the
semigroup is  the Laplacian  with domain
\begin{equation}
	D^{p}(\Lap_\theta)=\{u\in W^{2,p}(\Omega) \ : \ B_\theta(u) =
        0  \ 	\mbox{on $\partial\Omega$} \} 
      \end{equation}
and is  a sectorial operator, see  \cite{denk2004new}  and
      \cite{DdTRB23} for a simple proof when $p=2$. 
Furthermore, if $p=\infty$, as in \cite{lunardi} Corollary 3.1.21 and
Corollary 3.1.24, the generator is  the Laplacian  with domain
\begin{equation} \label{eq:domain_Linfty}
	D^\infty(\Lap_\theta)=\{u\in \bigcap_{p\geq
          1}W^{2,p}_{loc}(\adh{\Omega}) \ : \ u, \Lap u \in
        L^\infty(\Omega), \ B_\theta(u) = 0 \ 	\mbox{on $\partial\Omega$} \}
      \end{equation}
and is  also a sectorial operator with a non dense domain, so the
semigroup is analytic but not strongly continuous. Note that, by the Sobolev
embeddings, $D^\infty(\Lap_\theta)\subset
C^{1+\alpha}(\adh{\Omega})$ for any $\alpha \in (0,1)$.

For $1\leq p<\infty$ the semigroup above provides the unique solution
of \eqref{eq:heat_theta}, see e.g. Section 4.1 in \cite{pazy}. For
$p=\infty$ since the semigroup is not strongly continuous and the
domain \eqref{eq:domain_Linfty} is not dense, standard results do not
apply. Hence, we now give a uniqueness result.

\begin{theorem}
\label{thm:meaninglinf}
Let $u_0\in L^\infty(\Omega)$ and  $u(t):=S^\theta(t)u_0$. Then  $u\in
C^1((0,\infty),L^\infty(\Omega))$, for  $t>0$, $u(t)\in D^\infty(\Lap_\theta)$ as in
\eqref{eq:domain_Linfty} and  $\frac{du}{dt}(t)=\Lap u(t)$ and, as
$t\to 0^{+}$, 
\begin{equation}
u(t)\mywto{*} u_0 \qquad L^\infty(\Omega) . 
\end{equation}

Conversely if  $v\in C^1((0,\infty),L^\infty(\Omega))$ satisfies
$v(t)\in D^\infty(\Lap_\theta)$ and  $\frac{dv}{dt}(t)=\Lap v(t)$ for every $t>0$ and 
	$v(t)\mywto{*} u_0$, weak* in $L^\infty(\Omega)$, then we have that $v\equiv u$.
\end{theorem}
\begin{proof}
During the proof we suppress the superscript $\theta$ for
simplicity. The first statement is because the semigroup in
$L^{\infty}(\Omega)$ is analytic with the domain
$D^\infty(\Lap_\theta)$ in \eqref{eq:domain_Linfty}. The weak-*
convergence $u(t)\mywto{*} u_0$ in $L^\infty(\Omega)$ is a consequence
of duality, since for  $\varphi\in L^1(\Omega)$ 
\begin{equation}
\int_\Omega u(t) \varphi = \int_\Omega S(t)u_0 \varphi =
\int_\Omega u_0 S (t)\varphi \to \int_\Omega u_0 \varphi 
\end{equation}
when $t\to 0$  because $S(t)$ is a $C^0$ semigroup in $L^1(\Omega)$. 

For the converse, given $\varphi\in C^\infty_c(\Omega)$, define, for $t>0$,
\begin{equation}
	I(s):=\int_\Omega \varphi S(t-s)v(s)  = \int_\Omega v(s)
	S(t-s)\varphi \qquad s\in [0,t] . 
\end{equation}
Then $I(s)$ is a continuous function in $(0,t]$  because
$v\in C((0,\infty), L^\infty(\Omega))$ and
$S(t-\cdot)\varphi\in C([0,t],  L^1(\Omega))$. In addition, $I(s)$ is
continuous up to $0$  because,  when $s\to 0$, $S(t-s)\varphi\to S(t)\varphi$ in
$L^1(\Omega)$ and $v(s)\mywto{*}u_0$ in
$L^\infty(\Omega)$.

Now, we compute  the  derivative of $I(s)$ in
$(0,t)$. As $v\in C^1((0,\infty),L^\infty(\Omega))$ with
$\frac{dv}{dt}(t)=\Lap v(t)$, and at the same time,
$S(t-\cdot)\varphi\in C^1((0,t),L^1(\Omega))$ with
$\frac{d}{dt}S(t-s)\varphi=\Lap S(t-s)\varphi$ we have
\begin{equation}
	I'(s) = \int_\Omega \Lap v (s) S(t-s)\varphi-\int_\Omega v(s)
        \Lap S(t-s)\varphi . 
\end{equation}
Using that $\Lap$ commutes with the semigroup, as well as \eqref{eqn:Sexchange}, we obtain
\begin{equation}
	I'(s) = \int_\Omega \Lap S(t-s)v (s) \varphi -\int_\Omega  S(t-s)v(s) \Lap \varphi = 0
\end{equation}
where we have used the definition of the weak derivatives, as $\varphi\in C^\infty_c(\Omega)$.
 Therefore, $I(0)=I(t)$ so
\begin{equation}
	\int_\Omega v(t)\varphi = \int_\Omega S(t)u_0 \varphi, \quad
        t>0 
\end{equation}
and we obtain $v(t)=S(t)u_0$ almost everywhere for $t>0$.
\end{proof}

Moreover, the semigroup has an integral positive
kernel (see \cite[Theorem 2.5]{DdTRB23}), that is, $k^\theta: 
\Omega \times \Omega \times (0,\infty) \to
(0,\infty)$ such that for all $1\leq p\leq\infty$ and
$u_{0}\in L^{p}(\Omega)$,  
\begin{equation}
	\label{eqn:ackrnpre}
	S^\theta(t)u_0(x)=\int_\Omega k^\theta(x,y,t)u_0(y)dy
	, \qquad x\in \Omega, \quad t>0. 
\end{equation}
In addition, $k^\theta(x,y,t)=k^\theta(y,x,t)$, which reflects the
property \eqref{eqn:Sexchange}.

If we consider  $S^{\theta_1}(t)$ and $S^{\theta_2}(t)$   the
semigroups above for different $\theta$-boundary
conditions  we have (see \cite[Theorem 2.10]{DdTRB23}) that if $	0 \leq \theta_1\leq
\theta_2\leq 1 $ then for $u_{0}\geq 0$ we have
\begin{equation}
	\label{eqn:neugeqdir2}
S^{\theta_1}(t)u_0\leq	S^{\theta_2}(t)u_0 \quad t>0 , 
\end{equation}
or equivalently,  the 
corresponding heat kernels satisfy 
\begin{equation}
\label{eq:comparison_kernels_theta}
  0< k^{\theta_1}(x,y,t)\leq k^{\theta_2}(x,y,t) \qquad
	x,y\in \Omega,  \ t>0. 
      \end{equation}
In particular,  for any $\theta$-boundary conditions we have  Gaussian upper
bounds for the  heat kernel of the form
	\begin{equation}
		\label{eqn:gyryabound}
		0 < k^\theta(x,y,t) \leq C\frac{e^{-\frac{\abs{x-y}^2}{4ct}}}{t^{N/2}} \qquad
		x,y\in \Omega , \quad t>0 
              \end{equation}
for some  constants $c,C>0$, since they hold for Neumann boundary
conditions (see \cite{gyryathesis} and  also \cite{gyrya2011neumann} Theorem 3.10), that is
for $\theta \equiv 1$, and \eqref{eq:comparison_kernels_theta}, 
see \cite{DdTRB23} Section 2.

The bounds above imply, by using Proposition \ref{prop:convgauss} in
Appendix \ref{app:young}, the following result. 

\begin{corollary}
  \label{cor:LpLq_estimates}
  For any $u_{0}\in L^{p}(\Omega)$ and $1\leq p\leq q \leq \infty$ we
  have 
          \begin{equation}
     \label{eqn:LpLq_estimates_theta}
     \norm{S^\theta(t)u_0}_{L^{q}(\Omega)}\leq
     \frac{C}{t^{\frac{N}{2}(\frac{1}{p}- \frac{1}{q})}} \norm{u_0}_{L^{p}(\Omega)} \quad t>0 . 
   \end{equation}
\end{corollary}

Concerning regularity of solutions and the kernels, we  can state the
following results from \cite[Theorem 2.2]{DdTRB24a}.

\begin{theorem}
	\label{thm:prop2}
The  following properties hold true.  
	\begin{enumerate}
        \item
For $u_0\in L^p(\Omega)$, with $1\leq p\leq \infty$, 
    $u (x,t)= S^\theta(t)u_0(x)$ is a
    $C^{\infty}(\Omega\times(0,\infty))\cap C^{1}(\adh{\Omega}\times(0,\infty))$ solution of the heat
    equation, that is
    \begin{displaymath}
      \left\{
        \begin{aligned}
          u_t(x,t)-\Lap u(x,t) & = 0 &&
          \forall(x,t)\in\Omega\times(0,\infty)
          \\
          B_\theta(u)(x,t) & =0 && \forall x\in \partial \Omega, \
          \forall t>0 .
        \end{aligned}
      \right. 
    \end{displaymath}

  \item
    The integral kernel is analytic in time.
Furthermore, $k^\theta(\cdot,y, \cdot\cdot)$ belongs to
$C^{\infty}(\Omega\times(0,\infty))\cap C^{1}(\adh{\Omega}\times(0,\infty))$ and
satisfies the heat equation with $\theta$-boundary conditions for any fixed $y\in\Omega$.

	\end{enumerate}
\end{theorem}

Now we give some results on the asymptotic profile $\Phi$ mentioned in
the Introduction, which actually depends on the boundary conditions,
so we will denote hereafter as $\Phi^{\theta}$. This is a  harmonic function
in $\Omega$, $\Phi^{\theta}  \in C^2(\overline{\Omega})\cap
C^{\infty}(\Omega)$ and satisfies the boundary conditions 
$B_{\theta} (\Phi)\equiv 0$ on $\partial \Omega$. It can be constructed either
as the monotonically decreasing limit
\begin{equation}
  \label{eq:asymptotic_profile_parabolic}
	\Phi^{\theta} (x) =\lim_{t\to\infty} u(x,t; 1_\Omega) =
        \lim_{t\to\infty}  S^\theta(t)1_\Omega (x) \quad x\in \Omega, 
\end{equation}
that is, the solution of \eqref{eq:heat_theta}  with $u_{0} =
1_\Omega$, 
or as  the monotonically decreasing limit 
\begin{equation}
  \label{eq:asymptotic_profile_elliptic}
	\Phi^{\theta}  (x)=\lim_{R\to\infty} \phi^{\theta}_R(x) \quad x\in \Omega, 
\end{equation}
where $ \phi_R^{\theta}$ are harmonic in $ \Omega_R\defeq \Omega\cap B(0,R)$
and satisfy $B(\phi_{R}^{\theta})(x) =0$ for $x\in\partial \Omega$ and $
\phi_{R}^{\theta} (x) =1$ if $\abs{x}=R$, 
see  \cite{DdTRB23} Section 3. 

For integrable data this function determines the exact amount of mass
lost through the hole since  for $u_{0} \in
L^{1}(\Omega)$ the \emph{asymptotic mass} of the solution (that is the remaining mass
not lost through the boundary)
is given by
\begin{equation}
  \label{eq:asymptotic_mass}
  m_{u_0} \defeq \lim_{t\to\infty}\int_\Omega u(x,t)\, dx = \int_\Omega u_0(x) \Phi^{\theta}  (x) \,
  dx . 
\end{equation}
Of course,  for Neumann boundary conditions, that is $\theta\equiv 1$,
$\Phi^{1}  \equiv 1$ in any dimensions, hence no loss of mass at all for any solution. For Robin or
Dirichlet boundary conditions, if $N\leq 2$ then $\Phi^{\theta}  \equiv 0$. That is,
all mass is lost through the boundary. On the other hand, if $N \geq
3$, $\Phi^\theta \not\equiv 0$ so some mass remains, and we have the following estimates from \cite[Proposition 2.6]{DdTRB24a}

\begin{proposition}
	\label{prop:estpro}

Let $N\geq 3$  and $\theta \not\equiv 1$. Then, there exists $C>0$ such that
\begin{equation}
		\label{eqn:boundphi}
		1-\frac{C}{\abs{x}^{N-2}}\leq \Phi^{\theta}(x) \leq 1 \qquad \forall x\in \Omega.
	\end{equation}	
	In addition, for any multi-index $\abs{\beta}\neq 0$, if $\Phi^\theta \in C^{\abs{\beta}}(\adh{\Omega})$ (which is true if $\partial\Omega$ and $\theta$ are sufficiently regular), there exists $C_\beta>0$ such that
	\begin{equation}
		\label{eqn:boundpsi}
		\abs{D^\beta\Phi^\theta(x)}\leq \frac{C_{\beta}}{\abs{x}^{N-2+\abs{\beta}}} \qquad x\in \Omega.
	\end{equation}
\end{proposition}

The following lemma from \cite[Lemma 3.2]{DdTRB24a} states that the
$\theta$-heat kernel in $\Omega$ and the kernel in $\RN$ are similar in
$L^1(\Omega)$ when the source point $y$ is far away from the hole. The
Lemma is stated for $N\geq 3$. Notice that
\eqref{eqn:lemma:kerncoml1biseq3} below is also true for $N\leq 2$ but gives
no interesting  information because the asymptotic profile is $\Phi^0\equiv
0$. Also, recall that the heat kernel in $\R^{N}$ is given by $k_{\mathbb{R}^N}(x,y,t)= G(x-y,t)$. 

\begin{lemma}
	\label{lemma:kerncoml1bis}
Assume  $N\geq 3$ and let  $k^{\theta}(x,y,t)$ be the heat kernel for $\theta-$boundary conditions
and $k_{\mathbb{R}^N}(x,y,t)$ the heat kernel in the whole space. Then 
\begin{equation}
\label{eqn:lemma:kerncoml1biseq3}
\int_\Omega \abs{k^{\theta}(x,y,t)-k_{\mathbb{R}^N}(x,y,t)}dx\leq
2(1-\Phi^0(y))+\int_\hole k_\RN(x,y,t)dx \qquad y\in \Omega, 
\end{equation}
where $\Phi^0$ is the asymptotic profile of $\Omega$ for Dirichlet
boundary conditions. In particular 
	\begin{equation}
		\label{eqn:lemma:kerncoml1biseq1}
		\limsup_{t\to\infty}\int_\Omega
                \abs{k^{\theta}(x,y,t)-k_{\mathbb{R}^N}(x,y,t)}dx\leq
                2(1-\Phi^0(y)) \qquad y\in \Omega . 
	\end{equation}
	Furthermore, for all $x\in\Omega$, 
	\begin{equation}
		\label{eqn:lemma:kerncoml1biseq2}
		\limsup_{t\to\infty}\limsup_{\abs{y}\to
                  \infty}\int_\Omega
                \abs{k^{\theta}(x,y,t)-k_{\mathbb{R}^N}(x,y,t)}dx=
           \limsup_{\abs{y}\to
                  \infty}     \limsup_{t\to\infty}\int_\Omega
                \abs{k^{\theta}(x,y,t)-k_{\mathbb{R}^N}(x,y,t)}dx =0
                . 
	\end{equation}
\end{lemma}

\section{Bounded  very weak solutions}
\label{sec:bounded-very-weak_solutions}

In this section we give a suitable (very) weak meaning to bounded  solutions of
non homogeneous problems, see \eqref{eqn:defveryweakeq0} below, as
well as suitable comparison results for such solutions. These
comparison results will be a main tool for the main results in Section
\ref{sec:initial_data_Linfty}. Notice that  $\partial^D \Omega$,
$\partial^R\Omega$ and $\partial^N \Omega$ below are as defined in Section
\ref{sec:preliminares}.

We start with the following definition, which stems from formally
multiplying \eqref{eqn:defveryweakeq0} by $\varphi$ and integrating by
parts in $\Omega \times (0,T)$, see \eqref{eqn:weakfrompeq2} and
\eqref{eqn:weakfrompeq3} for the computation regarding boundary
terms.

\begin{definition}
\label{def:veryweak}

Let $\Omega\subset \RN$ be an exterior domain and $u_{0}\in
L^{\infty}(\Omega)$, $f\in L^1((0,T),L^\infty(\Omega))$, $g\in
L^1((0,T),L^\infty(\partial \Omega))$.

We say that $u\in L^\infty(\Omega \times(0,T))$ is a very weak solution of
	\begin{equation}
		\label{eqn:defveryweakeq0}
		\left\{
		\begin{aligned}
			u_t-\Lap u & = f  \qquad && \text{in} \ \Omega\times(0,T) \\
			B_\theta(u) & = g \qquad && \text{on} \ \partial \Omega \times (0,T) \\
			u(0) & = u_0 \qquad  && \text{in} \ \Omega 
		\end{aligned}
		\right.
\end{equation}
if  for any $0\leq \varphi\in C^{2,1}(\adh{\Omega}\times[0,T])$ such
that $B_\theta(\varphi)=0$, $\varphi(T)\equiv 0$ and 
$$\abs{\varphi(x,t)}+\abs{\nabla \varphi(x,t)}+\abs{\Lap \varphi(x,t)}+ \abs{\varphi_t(x,t)}
\leq C e^{-c  \abs{x}^2} \qquad \forall x\in \Omega, \ \forall t\in[0,T],$$
for some $C,c>0$, we have 
\begin{equation}
\label{eqn:defveryweak}
-\int_0^T\int_\Omega u\left(\varphi_t+\Lap \varphi\right)
= \int_0^T\int_\Omega f\varphi  + \int_\Omega
u_0\varphi(0) +\int_0^T\int_{\partial^{R,N}
  \Omega}\frac{g\varphi}{\sin\left(\frac{\pi}{2}\theta\right)} -
\int_0^T\int_{\partial^D \Omega}g\frac{\partial \varphi}{\partial n} 
\end{equation}
where we have denoted $\partial^{R,N}\Omega \defeq \partial^R\Omega\cup \partial^N \Omega=\{x\in\Omega:0<\theta(x)\leq 1\}$.

        Furthermore, we will say that $u$ is a very weak supersolution
(respectively subsolution) of \eqref{eqn:defveryweakeq0} if it
satisfies \eqref{eqn:defveryweak} as an inequality with $\geq$
(respectively $\leq$). 

\end{definition}

Clearly, a very weak solution of \eqref{eqn:defveryweakeq0} is both a
very weak super and a very weak sub solution. 

For our main result below we will  need the following lemma on smooth
solutions of the heat equation with homogenous boundary data. 

\begin{lemma}
\label{lemma:varphi}

Let $0\leq h\in C^\infty_c(\Omega\times(0,T))$ and consider the solution of the heat equation
\begin{equation}
  \label{eq:heat_no_homogeneous}
\left\{
\begin{aligned}
\varphi_t-\Lap \varphi & = h \qquad && \text{in} \ \Omega\times(0,T) \\
B_\theta(\varphi) & = 0 \qquad && \text{on} \ \partial \Omega \times (0,T) \\
\varphi(0) & = 0 \qquad  && \text{in} \ \Omega.
\end{aligned}
\right.
\end{equation}
Then, $ \varphi\in C^{2,1}(\adh{\Omega}\times[0,T])$, $\varphi\geq 0$ and there exist $C,c>0$ depending on $T$ and $h$ such that
\begin{equation}
	\label{eqn:lemmavarphieq1}
\abs{\varphi(x,t)} + \abs{\nabla \varphi(x,t)} + \abs{\Lap
  \varphi(x,t)}+\abs{\varphi_t(x,t)} \leq C e^{-c\abs{x}^2} \qquad \mbox{ $\forall x\in\Omega$,  $t\in (0,T)$}.
\end{equation}
\end{lemma}
\begin{proof}
  {\bf Step 1.}
We first derive some estimates on  the solutions of the homogeneous
problems.   For this, given $0\leq \psi\in C^\infty_c(\Omega)$ 
we prove  that there exist constants $C,c>0$ depending on the support
of $\psi$, $T$ and $\norm{\psi}_{L^\infty(\Omega)}$ such that for all
$x \in \overline{\Omega}$ 
 and $t\in (0,T)$ 
\begin{equation}
\label{eqn:varphieq1}
\abs{S^\theta(t)\psi(x)}\leq Ce^{-c\abs{x}^2}, \qquad
                \abs{\nabla S^\theta(t)\psi(x)}\leq
                C\frac{e^{-c\abs{x}^2}}{\sqrt{t}} \qquad 	\abs{\Lap S^\theta(t) \psi(x)}\leq C\frac{
          e^{-c\abs{x}^2}}{t} . 
\end{equation}
This can be obtained from the Gaussian estimates of the heat kernels
for short times. For example, from \cite[Theorem
2.2]{mora1983semilinear} we have $\abs{D^\beta k^\theta(x,y,t)}\leq
Ct^{-\frac{N+\abs{\beta}}{2}}e^{-c\frac{\abs{x-y}^2}{t}}$, when $t\in
(0,T)$ and $x,y\in\Omega$, where $D^\beta$ is a spatial derivative of
order $\abs{\beta}\leq 2$ Then, for small $\abs{x}$, we have, by
using Proposition \ref{prop:convgauss}, 
\begin{equation}
	\abs{D^\beta S^\theta(t)\psi(x)}\leq
        \frac{C}{t^{\abs{\beta}}}\abs{\int_\Omega
          t^{-\frac{N}{2}}e^{-c\frac{\abs{x-y}^2}{t}}\psi(y)dy} \leq
        \frac{C\norm{\psi}_{L^\infty(\Omega)}}{t^{\abs{\beta}}} . 
\end{equation}

Now, take  $D>0$ such that
$\text{supp}(\psi)\subset B(0,D/2)$. Then, for any $\abs{x}\geq D$ and
$\abs{y}\leq D/2$ we have $\abs{x-y}^2\geq \abs{x}^2/4$. Hence, for
$\abs{x}\geq D$ 
\begin{equation}
	\abs{D^\beta S^\theta(t)\psi(x)}\leq \frac{C}{t^{\abs{\beta}}}\abs{\int_\Omega t^{-\frac{N}{2}}e^{-c\frac{\abs{x-y}^2}{t}}\psi(y)dy} \leq \frac{Ce^{-\frac{\abs{x}^2}{8t}}}{t^{\abs{\beta}}}\abs{\int_\Omega t^{-\frac{N}{2}}e^{-c\frac{\abs{x-y}^2}{2t}}\psi(y)dy}\leq \frac{Ce^{-\frac{\abs{x}^2}{8T}}\norm{\psi}_{L^\infty(\Omega)}}{t^{\abs{\beta}}}
\end{equation}

\medskip 
\noindent {\bf Step 2.}
Now, we obtain the estimates for
$\varphi$ in \eqref{eq:heat_no_homogeneous}. By the variation of constants formula
\begin{equation}
	\varphi(x,t) = \int_0^t (S^\theta(t-s)h(s))(x)ds, \qquad x \in
        \Omega, \ t\in (0,T).
\end{equation}
 
 As $h\geq 0$ and $S^\theta(t)$ preserves the order, then $\varphi\geq
 0$. Now, from Step 1, we can find $C,c>0$ depending on
 $\norm{h}_{L^\infty(\Omega\times(0,T))}$, the support of $h$ and $T$ such that,
\begin{equation}
	\abs{\varphi(x,t)}\leq \int_0^t C e^{-c\abs{x}^2}ds \leq Ce^{-c\abs{x}^2} \qquad \forall x\in\adh{\Omega}, \qquad t\in (0,T)
\end{equation}
and
\begin{equation}
	\abs{\nabla \varphi(x,t)}\leq \int_0^t C \frac{e^{-c\abs{x}^2}ds}{\sqrt{t-s}} \leq Ce^{-c\abs{x}^2} \qquad \forall x\in\adh{\Omega}, \qquad t\in (0,T).
\end{equation}

For  the term $\Lap \varphi$, we can use the expression from \cite[Lemma 3.2.1]{henry},
\begin{equation}
	(-\Lap)\varphi(t) = h(t)-S^\theta(t)h(t)+\int_0^t  (-\Lap) S^\theta(t-s)(h(s)-h(t))ds
\end{equation}
so, using the smoothness of $h$, we obtain $h(s)-h(t)=(s-t)g(s)$
where $g\in C^\infty_c(\Omega\times[0,T])$  and $\norm{g}_{L^\infty(\Omega\times(0,T))}\leq
2\norm{h}_{C^{1}(\Omega\times(0,T))}$. Therefore, for $x\in\Omega$ and $t\in (0,T)$, using \eqref{eqn:varphieq1},
\begin{equation}
\abs{-\Lap \varphi(x,t)} \leq Ce^{-c\abs{x}^2}+\int_0^t C e^{-c\abs{x}^2}ds
 \leq Ce^{-c\abs{x}^2} 
\end{equation}
where $C,c>0$ depend on $T$, the support of $h$ and
$\norm{h}_{C^1(\adh{\Omega}\times (0,T))}$. The bound on $\varphi_t$
is immediately obtained from the equation and the bound on $\Lap
\varphi$. As the bounds on the derivatives are independent of $t$, we
have $\varphi\in C^{2,1}(\adh{\Omega}\times[0,T])$. 
\end{proof}

Now we present the comparison result for very weak solutions of
\eqref{eqn:defveryweakeq0}. 

\begin{theorem}
\label{thm:weakcomp}

Assume $u_{0}\in
L^{\infty}(\Omega)$, $f\in L^1((0,T),L^\infty(\Omega))$, $g\in
L^1((0,T),L^\infty(\partial \Omega))$. 

If $u_{0} \leq 0$, $f\leq 0$ and $g\leq 0$ and  $u\in
L^\infty(\Omega\times[0,T))$  is a very weak subsolution of
\eqref{eqn:defveryweakeq0},  then 
\begin{equation}
	u\leq 0 \ \text{a.e. in} \ \Omega\times[0,T).
\end{equation}
\end{theorem}
\begin{proof}
  Observe that from the assumption on the data, for any $\varphi$ as
  in Definition \ref{def:veryweak} we have  in \eqref{eqn:defveryweak} 
\begin{equation}
\label{eqn:weakcompeq1}
  -\int_0^T\int_\Omega u\left(\varphi_t+\Lap \varphi\right)\leq 0. 
\end{equation}
For this note that in the Dirichlet part of the boundary,
$\partial^{D}\Omega$,  $\varphi=0$ and since $\varphi \geq 0$ in
$\overline{\Omega} \times [0,T]$ then $\frac{\partial
  \varphi}{\partial n} \leq 0$ on $\partial^{D}\Omega$ in
\eqref{eqn:defveryweak} while $\sin(\frac{\pi}{2}\theta) \geq 0$ on
$\partial^{R,N}\Omega$. 

Now let  $0\leq h\in C^\infty_c(\Omega\times(0,T))$. Define
$\tilde{h}(x,t)\defeq h(x,T-t)$, and consider the function
$\tilde{\varphi}\geq 0$ associated to $\tilde{h}$ from Lemma
\ref{lemma:varphi}. Then, define
$\varphi(x,t)=\tilde{\varphi}(x,T-t)$. Hence, $\varphi\geq 0$
satisfies the backwards heat equation 
\begin{equation}
\left\{
\begin{aligned}
\varphi_t+\Lap \varphi & = -h \qquad && \text{in} \ \Omega\times(0,T) \\
B_\theta(\varphi) & = 0 \qquad && \text{on} \ \partial \Omega \times (0,T) \\
\varphi(T) & = 0 \qquad  && \text{in} \ \Omega
\end{aligned}
\right.
\end{equation}
as well as the conditions of $\varphi$ in Definition
\ref{def:veryweak}, from Lemma \ref{lemma:varphi}. Hence, we can use
it as a test function in \eqref{eqn:weakcompeq1} and then 
$\int_0^T\int_\Omega uh\leq 0$. As $0\leq h\in C^\infty_c(\Omega\times(0,T))$ was arbitrary, we obtain
$u\leq 0$. 
\end{proof}

We immediately get the following corollary.

\begin{corollary}
\label{cor:uniqueness_veryweak_sltns}

\begin{enumerate}
  \item
    Assume $0\leq u_{0}\in L^{\infty}(\Omega)$,
    $0 \leq f\in L^1((0,T),L^\infty(\Omega))$,
    $0\leq g\in L^1((0,T),L^\infty(\partial \Omega))$.

    If $u\in L^\infty(\Omega\times[0,T))$ is a very weak
    subsolution of \eqref{eqn:defveryweakeq0}, then
    \begin{equation}
      u\geq 0 \ \text{in} \ \Omega\times[0,T).
    \end{equation}

  \item
For  $u_{0}\in
L^{\infty}(\Omega)$, $f\in L^1((0,T),L^\infty(\Omega))$ and  $g\in
L^1((0,T),L^\infty(\partial \Omega))$ there exists at most a very weak
solution $u\in L^\infty(\Omega \times[0,T))$ of \eqref{eqn:defveryweakeq0} as in Definition
\ref{def:veryweak}. 

  \end{enumerate}
\end{corollary}

The next result  states that suitable  classical
solutions,  are very weak solutions \eqref{eqn:defveryweakeq0}.

\begin{theorem}
	\label{thm:weakfromp}
	Let $u\in C^1((0,T],L^\infty(\Omega))\cap
        C^{2,1}(\Omega\times(0,T])\cap C^{1}(\adh{\Omega}\times(0,T))\cap
        L^\infty(\Omega\times(0,T))$ be such that
	\begin{equation}
	\left\{
	\begin{aligned}
		u_t-\Lap u & = f  \qquad && \text{in} \  \Omega\times(0,T] \\
		B_\theta(u) & = g \qquad && \text{on} \ \partial \Omega \times (0,T]
	\end{aligned}
	\right.
	\end{equation}	
for some $f\in L^1((0,T),L^\infty(\Omega))$, $g\in
L^1((0,T),L^\infty(\partial \Omega))$ and  for some $u_0\in
L^\infty(\Omega)$, we have, as $t\to 0^+$, 
\begin{equation}
	u(t)\mywto{*}u_0 \qquad   \text{weak-* } L^\infty(\Omega) 
\end{equation}
and $u(t), \abs{\nabla u(t)}, \Lap u(t) \in L^{\infty}(\Omega)$ for
every $t>0$.

Then, $u$ satisfies in the very weak sense,
\begin{equation}
  \label{eq:full_nonhomogeneous}
	\left\{
	\begin{aligned}
		u_t-\Lap u & = f  \qquad && \text{in} \ \Omega\times(0,T] \\
		B_\theta(u) & = g \qquad && \text{on} \ \partial \Omega \times (0,T] \\
		u(0) & = u_0 \qquad  && \text{in} \ \Omega .
	\end{aligned}
	\right.
\end{equation}
\end{theorem}
\begin{proof}
  Let $\varphi$ as in Definition \ref{def:veryweak}. As $\varphi\in
C^{2,1}(\adh{\Omega}\times [0,T])$ and $\varphi$ and $\varphi_t$ decay exponentially
in space uniformly in $t\in[0,T]$, we have $\varphi\in
C([0,T],L^\infty(\Omega)) \cap C^1([0,T], L^1(\Omega))$. 
 In addition, as $u\in C^1((0,T], L^\infty(\Omega))$, we obtain that $t\mapsto \int_\Omega u(t)\varphi(t)$ belongs to
$C^1((0,T])$. Therefore, for any $\varepsilon>0$, since $\varphi(T)\equiv 0$,
	\begin{equation}
		\label{eqn:weakfrompeq1}
	 -\int_\varepsilon^T \int_\Omega u\varphi_t=-\int_\varepsilon^T \left(\frac{d}{dt}\int_\Omega u\varphi -\int_\Omega u_t\varphi\right)=\int_\Omega u(\varepsilon)\varphi(\varepsilon)+\int_\varepsilon^T \int_\Omega u_t\varphi.
	\end{equation}
Notice that we need to introduce $\varepsilon>0$ because
$\norm{u_t(t)}_{L^\infty(\Omega)}$ is not assumed to be integrable up
to $t=0$.

Now, let $R>0$ large and consider $\Omega_R = \Omega \cap
B(0,R)$. Then, for any $t>0$, using the regularity of $u(t)$ and
$\varphi(t)$ we obtain
\begin{equation}
	\label{eqn:weakfrompeqpre2}
	\int_{\Omega_R} u\Lap \varphi=\int_{\Omega_R} \Lap u
        \varphi+\int_{\partial \Omega} \left(u\frac{\partial
            \varphi}{\partial n}-\frac{\partial u}{\partial
            n}\varphi\right) + \int_{\partial B(0,R)}
        \left(u\frac{\partial \varphi}{\partial n}-\frac{\partial
            u}{\partial n}\varphi\right). 
\end{equation}
But, as $u(t),\nabla u(t) \in L^\infty(\Omega)$ and  $\varphi(t),
\nabla \varphi(t)$ decay exponentially in space at infinity, then we
have that $\int_{\partial
  B(0,R)} \left(u\frac{\partial \varphi}{\partial n}-\frac{\partial
    u}{\partial n}\varphi\right)\to 0$ when $R\to \infty$. Therefore, as
$u(t),\Lap u(t)\in L^\infty(\Omega)$, and $\varphi(t),\Lap\varphi(t)\in
L^1(\Omega)$,  we obtain 
\begin{equation}
	\label{eqn:weakfrompeq2}
	\int_\Omega u\Lap \varphi=\int_\Omega \Lap u \varphi+\int_{\partial \Omega} \left(u\frac{\partial \varphi}{\partial n}-\frac{\partial u}{\partial n}\varphi\right).
\end{equation}
Using now  $B_\theta(\varphi)=0$ and $B_\theta(u)=g$, one obtains
\begin{equation}
	\label{eqn:weakfrompeq3}
	\int_{\partial \Omega} \left(u\frac{\partial \varphi}{\partial n}-\frac{\partial u}{\partial n}\varphi\right) = -\int_{\partial^R \Omega}\frac{g\varphi}{\sin\left(\frac{\pi}{2}\theta\right)}+\int_{\partial^D \Omega}g\frac{\partial \varphi}{\partial n}.
\end{equation}
Hence, combining \eqref{eqn:weakfrompeq1}, \eqref{eqn:weakfrompeq2} and \eqref{eqn:weakfrompeq3}, we obtain
\begin{equation}
	-\int_\varepsilon^T\int_\Omega u\left(\varphi_t+\Lap \varphi\right)=
	\int_\varepsilon^T\int_\Omega (u_t-\Lap u)\varphi-\int_\Omega u(\varepsilon)\varphi(\varepsilon)-\int_\varepsilon^T\int_{\partial^{R,N} \Omega}\frac{g\varphi}{\sin\left(\frac{\pi}{2}\theta\right)}+\int_\varepsilon^T\int_{\partial^D \Omega}g\frac{\partial \varphi}{\partial n}.
\end{equation}
Note that $\int_\varepsilon^T\int_\Omega \varphi\Lap u $ is well defined because $\Lap u = u_t- f \in L^1_{loc}((0,T],L^\infty(\Omega))$.
Now, using the fact that $u_t-\Lap u=f\in
L^1((0,T),L^\infty(\Omega))$, that $u(t) \mywto{*} u_0$
$L^\infty(\Omega)$ with the weak-* topology, and the regularity of
$\varphi,g$, if we let $\varepsilon\to 0$, we obtain
\eqref{eqn:defveryweak} and then $u$ is a very weak solution of
\eqref{eq:full_nonhomogeneous}. 
\end{proof}

The next result states that semigroup solutions and the variation of
constants formula with bounded data are very weak solutions of
\eqref{eqn:defveryweakeq0} with zero boundary data. 

\begin{proposition}
\label{prop:easyweak}

\leavevmode
\begin{enumerate}
\item
If  $u_0\in L^\infty(\Omega)$, then  $u(t)=S^\theta(t)u_0$ satisfies
  \begin{equation}
    \left\{
      \begin{aligned}
        u_t-\Lap u & = 0  \qquad && \text{in} \ \Omega\times(0,\infty) \\
        B_\theta(u) & = 0 \qquad && \text{on} \ \partial \Omega \times (0,\infty) \\
        u(0) & = u_0 \qquad && \text{in} \ \Omega
      \end{aligned}
    \right.
  \end{equation}
  in the very weak sense.
        
\item
  If
  $f\in L^1((0,T),L^\infty(\Omega))\cap
  C^{\alpha}((0,T),L^\infty(\Omega))$, then
  $v(t)=\int_0^tS^\theta(t-s)f(s)ds$ satisfies
  \begin{equation}
    \label{eqn:easyweakeq3}
    \left\{
      \begin{aligned}
        v_t-\Lap v & = f  \qquad && \text{in} \ \Omega\times(0,T) \\
        B_\theta(v) & = 0 \qquad && \text{on} \ \partial \Omega \times (0,T) \\
        v(0) & = 0 \qquad && \text{in} \ \Omega
      \end{aligned}
    \right.
  \end{equation}
  in the very weak sense.
\end{enumerate}
\end{proposition}
\begin{proof}
(i) From Theorem \ref{thm:meaninglinf} and \ref{thm:prop2}, we have
that $u$ satisfies most of the conditions of Theorem
\ref{thm:weakfromp} with $f=0$ and $g=0$. Hence, we just need to check that
$\abs{\nabla u(t)}\in L^\infty(\Omega)$ for every $t>0$ to apply Theorem
\ref{thm:weakfromp}. Also as $u(t)\in
C^1(\adh{\Omega})$ we just need to prove the boundedness of $\nabla u(t)$
for large $\abs{x}$. Hence, take $\abs{x}$ large enough so that
$B(x,1)\subset \Omega$. Then, as we know that $u(t)\in
L^\infty(\Omega)$, using the Schauder estimates of Theorem
\ref{thm:schauderest} in the Appendix, we obtain, for $t>0$,
	\begin{equation}
		\abs{\nabla u(x,t)}\leq C(t)\norm{u(\cdot)}_{L^\infty(B(x,1)\times [t/2,t])}\leq C(t)\norm{u_0}_{L^\infty(\Omega)}.
	\end{equation}

\noindent (ii) From the classical theory of semigroups we have that $v\in\
C^1((0,T],L^\infty(\Omega))\cap C((0,T], D^\infty(-\Lap_\theta))$ and
satisfies $v_t-\Lap v = f$ in $(0,T)$ (See for example \cite[Theorem
4.3.4]{lunardi} or \cite[Lemma 3.2.1]{henry}). In addition, $v\to 0$
strongly in $L^\infty(\Omega)$ as $t\to 0^{+}$.

Let $\varphi$ in the conditions of
Definition \ref{def:veryweak}. Using the same arguments as in the
proof of Theorem \ref{thm:weakfromp}, for any $\varepsilon>0$ 
		\begin{equation}
			\label{eqn:propeasyweakeq1}
		-\int_\varepsilon^T\int_\Omega v\varphi_t = \int_\varepsilon^T \int_\Omega v_t\varphi+\int_\Omega v(\varepsilon)\varphi(\varepsilon)= \int_\varepsilon^T \int_\Omega \left(\Lap v+f\right)\varphi +\int_\Omega v(\varepsilon)\varphi(\varepsilon) .
	\end{equation}

Now we prove that $\norm{\nabla  v(t)}_{L^\infty(\Omega)}< \infty$ for
all $t\in (0,T)$. First, since  $v\in C((0,T], D^\infty(\Lap_\theta))$ and
$D^\infty(\Lap_\theta) \subset C^{1+\alpha}(\adh{\Omega})$ for
any $0<\alpha<1$, see \eqref{eq:domain_Linfty}, then it is enough to
obtain bounds on $\nabla v(t)$ for large $|x|$.  
For this, as by Theorem \ref{thm:meaninglinf} and \ref{thm:prop2}, for
fixed $s\in(0,T)$, $S^\theta(t-s)f(s)$ is a classical solution of the
heat equation for $t\in (s,T)$, we can take $|x|$ large  such that  $B(x,1)\subset
\Omega$ and  use the Schauder estimates from
Theorem \ref{thm:schauderest} with
$Q=B(x,1)\times[\frac{t-s}{2},t-s]$, to
obtain 
\begin{equation}
	\abs{\nabla S^\theta(t-s)f(s)(x)}\leq C\frac{\norm{f(s)}_{L^\infty(\Omega)}}{\min\left(1,\sqrt{t-s}\right)}
\end{equation}
where $C$ is independent of $f$ and $s$. Therefore, for any $0<t<T$
and $x\in \Omega$ such that $B(x,1)\subset \Omega$, 
\begin{equation}
	\abs{\nabla v(x,t)}\leq \int_0^t C\frac{\norm{f(s)}_{L^\infty(\Omega)}}{\min\left(1,\sqrt{t-s}\right)}ds<\infty
\end{equation}
due to the integrability of $f$ and its continuity in $0<t<T$ in
$L^{\infty}(\Omega)$.

From this and  \eqref{eq:domain_Linfty} we have   $v(t)\in
W^{2,p}_{loc}(\Omega)$ for any $1\leq p<\infty$, and $v(t),
\abs{\nabla v(t)}, \Lap v(t)\in L^\infty(\Omega)$, and using the same
arguments as in the proof of Theorem \ref{thm:weakfromp} (see
\eqref{eqn:weakfrompeqpre2} and \eqref{eqn:weakfrompeq2}), we obtain 
\begin{equation}
	\int_\Omega v(t)\Lap \varphi(t) = \int_\Omega \Lap v(t) \varphi(t)
\end{equation}
where the boundary term $\int_{\partial \Omega} \left(v\frac{\partial
    \varphi}{\partial n}-\frac{\partial v}{\partial n}\varphi\right)$
from \eqref{eqn:weakfrompeq3} has vanished because
$B_\theta(\varphi)=B_\theta(v)=0$.

In addition, as $v\in C((0,T],
D^{\infty} (-\Lap_{\theta}))$, so $\Lap v\in C((0,T], L^\infty(\Omega))$, we can integrate in time to obtain 
\begin{equation}
	\label{eqn:propeasyweakeq2}
	\int_\varepsilon^T \int_\Omega v\Lap \varphi = \int_\varepsilon^T \int_\Omega \varphi\Lap v .
\end{equation}

Hence, combining \eqref{eqn:propeasyweakeq1} and \eqref{eqn:propeasyweakeq2} we obtain
\begin{equation}
	-\int_\varepsilon\int_\Omega v\left(\varphi_t+\Lap \varphi\right)=\int_\varepsilon^T\int_\Omega \varphi\left(v_t-\Lap v\right)+\int_\Omega v(\varepsilon)\varphi(\varepsilon)=\int_\varepsilon^T\int_\Omega \varphi f +\int_\Omega v(\varepsilon)\varphi(\varepsilon)
\end{equation}
so taking the limit $\varepsilon\to 0$ and using that $f$ is
integrable up to $t=0$ by hypothesis, $v(\varepsilon)\to 0$ in $L^\infty(\Omega)$ when
$\varepsilon\to 0$ and $\varphi$ is a regular function, we get 
\begin{equation}
	\int_0^T\int_\Omega v\left(\varphi_t-\Lap \varphi\right)=\int_0^T\int_\Omega \varphi f
\end{equation}
and $v$ is the very weak solution of \eqref{eqn:easyweakeq3}.
\end{proof}

\section{Asymptotic behaviour for  initial data  in $L^\infty(\Omega)$}
\label{sec:initial_data_Linfty}

Our first  result shows that the solution of the heat equation in an
exterior domain is similar to the one in the whole space when we look
far away from the hole.

\begin{theorem}
  \label{thr:dataLinfty_far_from_hole}
  
Assume $N\geq 3$ and
  $u_0\in L^\infty(\Omega)$ which we assume extended to $\R^{N}$ by
  zero outside $\Omega$. Let $u(t)=S^\theta(t)u_0$ be the solution
  of the heat equation for some homogeneous $\theta-$boundary
  conditions, and let $u_{\RN}(t)=S_\RN(t) u_0$ be the
  solution in the whole space for the same initial data (extended by
  zero outside $\Omega$). Then, given $\varepsilon>0$, there exists
  $R>0$ such that, for any $\abs{x}\geq R$,
	\begin{equation}
		\label{eqn:thmlinfeq1}
		\abs{u(x,t)-u_{\RN}(x,t)}\leq \varepsilon, \qquad t>0.
	\end{equation}
\end{theorem}
\begin{proof}
Using the symmetry of the kernels (see below~\eqref{eqn:ackrnpre}) and the estimate
\eqref{eqn:lemma:kerncoml1biseq3}, we have 
	\begin{equation}
		\label{eqn:thmlinfeq2}
		\begin{aligned}
		&	\abs{u(x,t)- u_{\RN}(x,t)}  \leq \int_\Omega
                        \abs{ k_\RN(x,y,t)-k^\theta(x,y,t)}
                        \abs{u_0(y)}dy \leq
                       \\
                     &  \norm{u_0}_{L^\infty(\Omega)}  \int_\Omega \abs{k_\RN(y,x,t)-k^\theta(y,x,t)}dy 
			\myleq{\eqref{eqn:lemma:kerncoml1biseq3}}
			\norm{u_0}_{L^\infty(\Omega)}\left(2(1-\Phi^0(x))+\int_\hole k_\RN(y,x,t)dy\right)
		\end{aligned}
	\end{equation}

Now, we have $\int_\hole k_\RN(y,x,t)dy\leq (4\pi t)^{-N/2}\abs{\hole}$, so we can find $T>0$ such that
\begin{equation}
	\label{eqn:thmlinfeq3}
	\norm{u_0}_{L^\infty(\Omega)} \int_\hole k_\RN(y,x,t) \, dy \leq \varepsilon/2, \qquad  x\in \Omega, \  t\geq T.
\end{equation}
In addition, for $t\leq T$, and $d(x,\partial\Omega)\geq R$, we have
$\int_\hole k_\RN(y,x,t)\leq e^{-\frac{R^2}{8t}} \int_{\hole} (4\pi
t)^{-N/2}e^{-\frac{\abs{x-y}^2}{8t}} \, dy \leq 2^{N/2} e^{-\frac{R^2}{8T}}$
so choosing $R$ large enough we have, 
\begin{equation}
	\label{eqn:thmlinfeq4}
	\norm{u_0}_{L^\infty(\Omega)}\int_\hole k_\RN(y,x,t) \, dy
        \leq \varepsilon/2, \qquad d(x,\partial\Omega)\geq R, \ t\leq
        T. 
\end{equation}

In addition, as $\Phi^0(x)\to 1$ when $\abs{x}\to \infty$, we can choose $R$ large enough such that
 \begin{equation}
 	\label{eqn:thmlinfeq5}
 	2\norm{u_0}_{L^\infty(\Omega)}(1-\Phi^0(x))\leq \varepsilon/2, \qquad d(x,\partial\Omega)\geq R.
 \end{equation}
Therefore, combining  \eqref{eqn:thmlinfeq2}, \eqref{eqn:thmlinfeq3}, \eqref{eqn:thmlinfeq4} and \eqref{eqn:thmlinfeq5}, we obtain \eqref{eqn:thmlinfeq1} for every $x\in \Omega$ satisfying $d(x,\partial \Omega)\geq R$.
\end{proof}

Now we prove that solutions in exterior domains in $L^\infty(\Omega)$
asymptotically converge to the solutions in $\RN$ times the asymptotic
profile $\Phi^{\theta}$.

\begin{theorem}
\label{thm:dataLinfty_Dirichlet}

Assume $N\geq 3$ and  $u_0\in L^\infty(\Omega)$. Let $u(t)=S^\theta(t)u_0$ be the solution for the heat equation with homogeneous $\theta-$boundary conditions, and $u_{\RN}(t)=S_\RN(t)u_0$ the solution of the heat equation in $\RN$ with $u_0$ extended by zero outside $\Omega$. Then
\begin{equation}
	\label{eqn:linfeq9}
\lim_{t\to\infty}\norm{u(t)-\Phi^{\theta}(\cdot)u_\RN(t)}_{L^\infty(\Omega)}=0  
\end{equation}
and moreover 
\begin{equation}
  \label{eq:rate_of_convergence_Linfty}
\norm{u(t)-\Phi^{\theta}(\cdot)u_\RN(t)}_{L^\infty(\Omega)} \leq \left\{
\begin{aligned}
& \frac{C\log(t)}{\sqrt{t}}\norm{u_0}_{L^\infty(\Omega)} &&
\text{if}  \ N=3 \\
& \frac{C}{\sqrt{t}}\norm{u_0}_{L^\infty(\Omega)}  && \text{if}  \ N>3
\end{aligned}
\right.
\end{equation}
for large $t>0$.
\end{theorem}
\begin{proof} 
\textbf{Step 1. }
  First to  remove any singular behavior near $t=0$ we   consider  a cut-off function $\eta\in
  C^\infty_c(\adh{\Omega}\times[0,\infty))$ such that $\eta\equiv 1$
  in $B(0,R)\times[0,1]$ with $R>0$ large enough so that
  $\hole\subset B(0,R)$, and $\eta\equiv 0$ in
  $\adh{\Omega}\times[2,\infty)$. Then, define  $v(x,t)\defeq 
S^{\theta}(t)u_0(x)-(1-\eta(x,t))\Phi^{\theta}(x)S_\RN(t)u_0(x)$. We have, by
Theorem \ref{thm:meaninglinf} and \ref{thm:prop2} and the regularity
of $\Phi^{\theta}$, that $v\in
C^1((0,\infty),L^\infty(\Omega))\cap C^{2,1}(\Omega\times(0,\infty))\cap C^1(\adh{\Omega}\times(0,\infty))\cap L^\infty(\Omega\times (0,\infty))$ and
satisfies in a pointwise sense 
\begin{equation}
      \label{eqn:linfeq1}
  \left\{
		\begin{aligned}
			v_t-\Lap v & = f \qquad
			&& \text{in} \ \Omega\times(0,\infty) \\
			B_\theta(v) & = \sin \left(\frac{\pi}{2}\theta\right)g \qquad && \text{on} \ \partial\Omega\times(0,\infty) \\
			v(t) & \mywto{*}  \widetilde{u}_0, \qquad \qquad t\to 0^+, && L^\infty(\Omega) \ \ \text{weak}-*
		\end{aligned}
		\right.
	\end{equation}
where 
\begin{equation}
	\label{eqn:linfeq6}
	\begin{aligned}
		f(x,t) & \defeq 2\nabla ((1-\eta)\Phi^\theta)\cdot \nabla S_\RN(t)u_0 - (\eta_t-\Lap \eta)\Phi^\theta S_\RN(t)u_0 \\
		g(x,t)  & \defeq \Phi^\theta\left((\eta-1)\frac{\partial S_\RN(t)u_0}{\partial n}+\frac{\partial \eta}{\partial n}S_\RN(t)u_0\right) \\
		\widetilde{u}_0(x) & =
                (1-(1-\eta(x,0))\Phi^{\theta}(x))u_0(x) . 
	\end{aligned}
\end{equation}

From  the explicit form of the heat kernel in $\RN$, we have that
\begin{equation}
	\label{eqn:linfeq4}
	\norm{\nabla S_\RN(t)u_0}_{L^\infty(\Omega)}\leq
	\frac{C}{\sqrt{t}}\norm{u_0}_{L^\infty(\Omega)} . 
\end{equation}
Therefore, as $\Phi^\theta, \nabla \Phi^\theta\in L^\infty(\Omega)$, $\norm{S_\RN(t)u_0}_{L^\infty(\Omega)}\leq \norm{u_0}_{L^\infty(\Omega)}$ and $\eta\in C^\infty_c(\adh{\Omega}\times [0,\infty))$, from \eqref{eqn:linfeq6} and \eqref{eqn:linfeq4} we obtain 
\begin{equation}
	\label{eqn:linfeq5}
	\norm{f(t)}_{L^\infty(\Omega)}+\norm{g(t)}_{L^\infty(\partial \Omega)}\leq \frac{C}{\sqrt{t}}\norm{u_0}_{L^\infty(\Omega)} \qquad t>0.
\end{equation}
 In particular, for every $T>0$, $f\in L^1((0,T), L^\infty(\Omega))$ and $g\in L^1((0,T),L^\infty(\partial \Omega))$. 
Hence, from  Theorem \ref{thm:weakfromp},   \eqref{eqn:linfeq1} is
satisfied in the  very weak sense of  Definition \ref{def:veryweak}. 

\noindent \textbf{Step 2. }
In this step we will construct two auxiliary functions $w$ and $y$
which we will use to estimate  $v$. First, as for any $T>0$, $f\in L^1((0,T), L^\infty(\Omega))$, and, from the expression \eqref{eqn:linfeq6} and Theorem
\ref{thm:meaninglinf}, we obtain $f\in
C^1((0,\infty),L^\infty(\Omega))$,
we can consider
\begin{equation}
	w(t) = S^{\theta}(t)\widetilde{u}_0 + \int_0^t S^{\theta}(t-s)f(s) ds, \qquad t>0, 
\end{equation}
which, by Proposition \ref{prop:easyweak}, satisfies in the very weak sense
\begin{equation}
	\label{eqn:linfeq2}
	\left\{
	\begin{aligned}
		w_t-\Lap w & = f  \qquad && \text{in} \ \Omega\times(0,\infty) \\
		B_\theta(w) & = 0 \qquad && \text{on} \ \partial \Omega \times (0,\infty) \\
		w(0) & = \widetilde{u}_0 \qquad  && \text{in} \ \Omega .
	\end{aligned}
	\right.
\end{equation}

Now, note that, as $\eta\equiv 1$ in a neighbourhood of $\partial
\Omega$ for $t\in [0,1)$, from \eqref{eqn:linfeq6}, we can
improve the estimate \eqref{eqn:linfeq5} to  
\begin{equation} \label{eq:decay_of_h}
	\norm{g(t)}_{L^\infty(\partial \Omega)}\leq h(t)\defeq
        \frac{C}{(1+t)^{1/2}}\norm{u_0}_{L^\infty(\Omega)}, \qquad
        t>0. 
\end{equation}
Finally, consider $\psi=D(1-\Phi^0)$ with $D>0$. As $\psi$ is harmonic and attains its maximum value on the boundary, from the Hopf lemma we know
that $\frac{\partial \psi}{\partial n}>0$ on $\partial \Omega$, so we
can find $D>0$ large enough such that $B_\theta(\psi)\geq
\sin\left(\frac{\pi}{2}\theta\right)$ on $\partial \Omega$. Then, we define 
\begin{equation} \label{eq:definition_y}
	y(t)\defeq h(t)\psi-\int_0^t
        S^{\theta}(t-s)(h'(s) \psi)ds ,
\end{equation}
which, as we will prove below, satisfies in the very weak sense,
\begin{equation}
	\label{eqn:linfeq3}
	\left\{
	\begin{aligned}
		y_t-\Lap y & = 0 \quad && \text{in} \ \Omega\times(0,\infty)
		\\
		B_\theta(y) & \geq
		\sin\left(\frac{\pi}{2}\theta\right)h(t)\geq
		\sin\left(\frac{\pi}{2}\theta\right)g(t)\qquad && \text{on} \ 
		\partial\Omega \times (0,\infty)
		\\
		y(0) & \geq 0 \qquad && \text{in} \  \Omega.
	\end{aligned}
	\right.
\end{equation}

Assumed this for a moment,  we obtain then that $-y + w \leq v\leq w+y$. 
Actually, if we define $z\defeq v-w-y$, combining \eqref{eqn:linfeq1},
\eqref{eqn:linfeq2} and \eqref{eqn:linfeq3}, we obtain that $z$
satisfies in the  very weak sense 
\begin{equation}
	\label{eq:z=v-w}
	\left\{
	\begin{aligned}
		z_t-\Lap z  & = 0 \qquad
		&& \text{in} \ \Omega\times(0,\infty) \\
		B_\theta(z) & \leq  0 \qquad && \text{on} \ \partial\Omega\times(0,\infty) \\
		z(0) & \leq 0 \qquad && \text{in} \ \Omega.
	\end{aligned}
	\right.
\end{equation}
Therefore, using Theorem \ref{thm:weakcomp}, we obtain that $z\leq 0$,
that is $v\leq w+y$, as claimed. The other inequality is obtained in
the same way.

Hence, let us prove that \eqref{eqn:linfeq3} is satisfied in the very
weak sense. Let us start with $y_1(t)\defeq h(t)\psi$, which belongs
to $ C^1((0,\infty),L^\infty(\Omega))\cap
C^{2,1}(\adh{\Omega}\times(0,\infty))$, satisfies in a pointwise sense 
\begin{equation}
	\label{eqn:linfeq7}
	\left\{
	\begin{aligned}
		(y_1)_t-\Lap y_1 & = h'(t)\psi \qquad && \text{in} \  \Omega\times(0,\infty) \\
		B_\theta(y_1) & = h(t)B_\theta(\psi) \qquad && \text{on} \  \partial \Omega \times (0,\infty) \\
		y_1(0) & = h(0)\psi \qquad && \text{on} \ \Omega
	\end{aligned}
	\right.
\end{equation}	
and $y_1(t)\to h(0)\psi\geq 0$ in $L^\infty(\Omega)$ when $t\to
0$. Therefore, $y_1$ is under the conditions of Theorem
\ref{thm:weakfromp} and \eqref{eqn:linfeq7} is satisfied in the very weak sense.

As for  $y_2(t)\defeq -\int_0^t S^\theta(t-s)\left( h'(s) \psi \right)ds$, 
we have that, for every $T>0$, $h'\psi \in
C^1((0,T),L^\infty(\Omega))\cap
L^1((0,T),L^\infty(\Omega))$. Therefore, $y_2$ is under the conditions
of Proposition \ref{prop:easyweak} (ii) and satisfies in the very weak sense
\begin{equation}
		\label{eqn:linfeq8}
	\left\{
	\begin{aligned}
		(y_2)_t-\Lap y_2 & = -h'(t)\psi \qquad && \text{in} \  \Omega\times(0,T] \\
		B_\theta(y_2) & = 0 \qquad && \text{on} \  \partial \Omega \times (0,T] \\
		y_2(0) & = 0 \qquad && \text{in} \ \Omega . 
	\end{aligned}
	\right.
\end{equation}
Combining \eqref{eqn:linfeq7} and \eqref{eqn:linfeq8}, we obtain that 
\eqref{eqn:linfeq3} is satisfied in the very weak sense.  

In the next steps we will show that $w(t)$ and $y(t)$ tend to zero
uniformly as $t\to \infty$. 

\noindent \textbf{Step 3.}
Let us see that $w(t)\to 0$ in $L^\infty(\Omega)$ when
$t\to\infty$. First, by Proposition \ref{prop:estpro}, we have
$\Phi^\theta\in L^\infty(\Omega)$ and the decay at infinity of
$\Phi^{\theta}$ implies that $(1-\Phi^\theta)\in
L^{\frac{N}{N-2},\infty}(\Omega)$, see Appendix \ref{app:young}. Therefore, as $\eta\in
C^\infty_c(\adh{\Omega}\times[0,\infty))$, we have  from \eqref{eqn:linfeq6}
\begin{equation}
	\label{eqn:linfeq10}
	\norm{\widetilde{u}_0}_{L^\infty(\Omega)}+\norm{\widetilde{u}_0}_{L^{\frac{N}{N-2},\infty}(\Omega)}\leq
        C\norm{u_0}_{L^\infty(\Omega)} . 
\end{equation}
Then, using the Gaussian estimates \eqref{eqn:gyryabound} as well as
the properties of convolution in Lorentz spaces in   Proposition
\ref{prop:convgauss} and the fact that $S^\theta(t)$ is a semigroup of
contractions in $L^{\infty}(\Omega)$, we get for any $t>0$ 
\begin{equation}
	\label{eqn:limlinfteq1_0}
		\|S^{\theta} (t)\widetilde{u}_0\|_{L^\infty(\Omega)}    \leq
		C\min\left(\frac{\norm{\widetilde{u}_0}_{L^{\frac{N}{N-2},\infty}}}{t^{\frac{N-2}{2}}}
		,
		\norm{\widetilde{u}_0}_{L^\infty(\Omega)}\right)\myleq{\eqref{eqn:linfeq10}} 
		\frac{C\norm{u_0}_{L^\infty(\Omega)}}{(1+t)^{\frac{N-2}{2}}}.
\end{equation}

Now, let us estimate $f(t)$ in \eqref{eqn:linfeq6}. By Proposition
\ref{prop:estpro}, we have $\Phi^\theta\in L^\infty(\Omega)$ and again
the decay at infinity implies 
$\nabla \Phi^\theta\in L^\infty(\Omega)\cap
L^{\frac{N}{N-1},\infty}(\Omega)$, so using \eqref{eqn:linfeq4} and
the fact that $\eta$ and $\eta_{t}- \Delta \eta$ are of compact
support, we
have  
\begin{equation}
	\label{eqn:limlinfteq5}
	\norm{f(t)}_{L^\infty(\Omega)}+\norm{f(t)}_{L^{\frac{N}{N-1},\infty}(\Omega)}\leq \frac{C \norm{u_0}_{L^\infty(\Omega)}}{\sqrt{t}} \qquad t>0.
\end{equation}

Therefore, using that $S^\theta(t)$ is a semigroup of contractions in $L^\infty(\Omega)$,
\begin{equation}
	\label{eqn:limlinfteq2_0}
		\int_{t-1}^t \|S^{\theta}(t-s)  f(s)\|_{L^\infty(\Omega)} \, ds  \myleq{\eqref{eqn:limlinfteq5}} \int_{t-1}^t
		\frac{C\norm{u_0}_{L^\infty(\Omega)} }{\sqrt{s}}  \, ds  
			 \leq
                         \frac{C\norm{u_0}_{L^\infty(\Omega)}}{\sqrt{t}}
                         . 
\end{equation}

On the other hand,  using the Gaussian estimates \eqref{eqn:gyryabound} and  Proposition
\ref{prop:convgauss} with $p=\frac{N}{N-2}$ and $q=\infty$, we obtain
\begin{equation}
\label{eqn:limlinfteq3_0}
\begin{aligned}
& \int_{0}^{t-1} \norm{S^{\theta}(t-s)f(s)}_{L^\infty(\Omega)} \, ds   \leq
\int_0^{t-1}\frac{C\norm{f(s)}_{L^{\frac{N}{N-1},\infty}(\Omega)}}{(t-s)^{\frac{N-1}{2}}}
 \, ds  \leq
 \int_0^{t-1}\frac{C\norm{u_0}_{L^\infty(\Omega)}}{\sqrt{s} (t-s)^{\frac{N-1}{2}}}
 \, ds \\ 
& \leq
\frac{C}{t^{\frac{N-1}{2}}}
\int_0^{t/2}\frac{\norm{u_0}_{L^\infty(\Omega)}}{\sqrt{s}} \, ds+
\frac{C}{\sqrt{t}}
\int_{t/2}^{t-1}\frac{\norm{u_0}_{L^\infty(\Omega)}}{(t-s)^{\frac{N-1}{2}}}
\, ds 
\leq \left\{
\begin{aligned}
	C\left(\frac{1}{\sqrt{t}}+\frac{\log(t)}{\sqrt{t}}\right)\norm{u_0}_{L^\infty(\Omega)}
        \quad  & \text{if} \ N=3 \\ 
	C\left(\frac{1}{t^{\frac{N-2}{2}}}+\frac{1}{\sqrt{t}}\right)\norm{u_0}_{L^\infty(\Omega)}
	\quad & \text{if}   \ N>3 
\end{aligned}
\right.
\end{aligned}
\end{equation}
for large $t>0$. 

Therefore, combining \eqref{eqn:limlinfteq1_0},
\eqref{eqn:limlinfteq2_0} and \eqref{eqn:limlinfteq3_0}, and using
$\frac{N-2}{2}\geq \frac{1}{2}$ 
when $N> 3$, we have that
$\norm{w(t)}_{L^\infty(\Omega)}\to 0$ when $t\to \infty$ with the
 rate in \eqref{eq:rate_of_convergence_Linfty}. 

\noindent \textbf{Step 4. }Let us see that $y(t)\to 0$ in $L^\infty(\Omega)$ when
$t\to\infty$ even faster than $w(t)$ in Step 3. Once again from the decay at infinity  in Proposition
\ref{prop:estpro}, we have  that $\psi=D(1-\Phi^0)\in L^\infty(\Omega)\cap
L^{\frac{N-2}{2}, \infty}(\Omega)$ and therefore, using the same
arguments in  Step 3 to prove \eqref{eqn:limlinfteq1_0}, we obtain 
$\norm{S^{\theta}(t)\psi}_{L^\infty(\Omega)}\leq
\frac{C}{(1+t)^{\frac{N-2}{2}}}$. Thus, for large $t>0$, using
\eqref{eq:decay_of_h}, we have 
\begin{equation}
	\begin{aligned}
		& \bnorm{\int_0^t S^{\theta}
			(t-s)( h'(s) \psi ) \, ds}_{L^\infty(\Omega)}  \leq
		\int_0^t\frac{C\norm{u_0}_{L^\infty(\Omega)}}{(1+t-s)^{\frac{N-2}{2}}(1+s)^{3/2}}
                \, ds \\
		& \leq C \norm{u_0}_{L^\infty(\Omega)}\left(
                  \frac{1}{(1+\frac{t}{2})^{\frac{N-2}{2}}}\int_0^{t/2}\frac{ds}{(1+s)^{3/2}}
                  + \frac{1}{(1+\frac{t}{2})^{\frac{3}{2}}} \int_{t/2}^t\frac{ds}{(1+t-s)^{\frac{N-2}{2}}}
                \right)		\leq
                \frac{C\norm{u_0}_{L^\infty(\Omega)}}{(1+t)^{\frac{1}{2}}}                
	\end{aligned}
      \end{equation}
where we have used that, since $N\geq 3$,   $\frac{N-2}{2}\geq \frac{1}{2}$ 
      
Therefore, from this,   \eqref{eq:definition_y} and
\eqref{eq:decay_of_h}, we have, for large $t>0$, $\norm{y(t)}_{L^\infty(\Omega)}\leq
C\norm{u_0}_{L^\infty(\Omega)}/\sqrt{t}$.

\noindent \textbf{Step 5. } 
From Step 2,  $-y + w \leq v\leq w+y$ while from 
Steps 3 and 4, $\norm{w(t)}_{L^\infty(\Omega)} +
\norm{y(t)}_{L^\infty(\Omega)}$ decays to zero with the rate in
\eqref{eq:rate_of_convergence_Linfty} and we get the result. 
\end{proof}

The solutions of the heat equation in $\R^{N}$ present a complex
dynamical behaviour as shown in  \cite{vazquez2002complexity}. The
previous result allows us to translate such complex behavior to the
solutions of the heat equation in an exterior domain.

\begin{theorem}
\label{thm:compleq1}
Assume $N\geq 3$ and some homogeneous $\theta-$boundary conditions. There exists $u_0\in L^\infty(\Omega)$ such that, for
every $\varepsilon,\delta, T>0$ and $g\in L^\infty(\RN)$ with
$\norm{g}_{L^\infty(\RN)}\leq 1$, there exists $t_{*}\geq T$ such
that, if we denote $u(x,t)=S^\theta(t)u_0(x)$ and $\tilde{g}=S_\RN(1)g$, 
\begin{equation}
\label{eqn:compleq1}
		\abs{u(x,t_{*})-\Phi^\theta(x)\tilde{g}\left(\frac{x}{\sqrt{t_{*}}}\right)}\leq
                \varepsilon \qquad \forall \abs{x}^2\leq \delta t_{*} . 
	\end{equation}
\end{theorem}
\begin{proof}
  Given $u_0\in L^\infty(\Omega)$ let us denote
  $u_\RN(x,t)=S_\RN(t)u_0$, where $u_0$ is extended by zero outside
  $\Omega$.  It was proved in \cite{vazquez2002complexity} Theorem 2.1
  that the set of accumulation points of
  $u_{\R^{N}}(\sqrt{t} \cdot , t)$ in $L^{\infty}_{loc}(\R^{N})$, as
  $t\to \infty$ coincides with $S_{\R^{N}}(1)\phi$, where $\phi$
  ranges over $\omega(u_0)$, the set of accumulation points of
  $u_0(\lambda \cdot)$ in $L^\infty(\RN)$ with the weak-* topology, as
  $\lambda \to \infty$. 
	
  In addition, for any bounded sequence of functions in
  $L^\infty(\RN)$, it was proved in \cite{vazquez2002complexity} that
  there exists $u_0\in L^\infty(\RN)$ such that $\omega(u_0)$ contains
  this sequence. As the space $L^\infty(\RN)$ is separable with the
  weak-* topology, and $\omega(u_0)$ is weak-* closed, we can find $u_0$
  such that $\omega(u_0)$ contains the unit ball
  $B_{L^\infty(\RN)}$. Therefore, given $g\in B_{L^\infty(\RN)}$, we
  have that $\tilde{g} =S_\RN(1)g$ is an accumulation point of
  $u_{\R^{N}}(\sqrt{t} \cdot , t)$ in
  $L^{\infty}_{loc}(\R^{N})$ as $t\to \infty$. Hence, given $\delta,T>0$, there exists
  $t_{*}\geq T$ such that
	\begin{equation}
		\abs{u_{\R^{N}}(\sqrt{t_{*}} y , t_{*})-\tilde{g}(y)}\leq \varepsilon/2 \qquad \forall \abs{y}^2\leq \delta
	\end{equation}
so changing variables $x=\sqrt{t_{*}} y$ we obtain
	\begin{equation}
		\label{eqn:compleq3}
	\abs{u_{\R^{N}}(x ,
          t_{*})-\tilde{g}\left(\frac{x}{\sqrt{t_{*}}}\right)}\leq
        \varepsilon/2 \qquad \forall \abs{x}^2\leq \delta t_{*} .
	\end{equation}

        In addition, using Theorem \ref{thm:dataLinfty_Dirichlet} we have
that, taking $T$ larger if necessary, 
\begin{equation}
	\label{eqn:compleq2}
	\norm{S^\theta(t)u_0-\Phi^\theta(\cdot)u_\RN(\cdot,t)}_{L^\infty(\Omega)}\leq \varepsilon/2 \qquad \forall t\geq T.
\end{equation}
Combining then \eqref{eqn:compleq2} and \eqref{eqn:compleq3}, as well
as the fact that $0\leq \Phi^{\theta}\leq 1$ one obtains \eqref{eqn:compleq1}. 
\end{proof}

Another expression of the complexity of the behaviour for bounded
initial data is the following result that states that,  given a bounded sequence of positive values
$\{a_{n}\}_{n}$ we can construct an initial data
$u_0\in L^\infty(\Omega)$ such that at any given point $x_{0}\in \Omega$,
$S^{\theta}(t)u_{0}$ attains the values $a_{n}$ at a sequence of divergent times
$t_{n}$.

\begin{theorem}
\label{thm:dowhatyouwant}

Let $\{a_n\}_{n\in\mathbb{N}}$ be a sequence of real numbers such that
$0<a_n<1$ for every $n\in \mathbb{N}$. Consider some homogeneous $\theta-$boundary conditions. Then, for any $x_0\in \Omega$,
there exists an initial datum $u_0\in L^\infty(\Omega)$ with
$\norm{u_0}_{L^\infty(\Omega)} = 1$ and a sequence of times
$t_n\to \infty$ such that
\begin{equation}
S^\theta(t_n)u_0(x_0)=a_n \Phi^\theta(x_0) . 
\end{equation}
\end{theorem}

Prior to prove the result we will prove the following auxiliary lemma. We will use the notation $B_R\defeq B(0,R)$.

\begin{lemma}
\label{lemma:predowhatyouwant}
Given $\varepsilon>0$, $x_0\in \RN$, $T>0$ and $R>0$, there exists $t>T$ and $\tilde{R}>R$ such that
\begin{equation}
\int_{B_R}k_{\RN}(x_0,y,t) \,  dy +\int_{\RN\backslash
  B_{\tilde{R}}} k_\RN(x_0,y,t) \,  dy  \leq \varepsilon
\end{equation}
or equivalently
\begin{equation}
\int_{B_{\tilde{R}}\backslash B_{R}} k_\RN(x_0,y,t) \,  dy \geq 1-\varepsilon
. 
\end{equation}
\end{lemma}
\begin{proof}
  As $k_\RN(x_0,y,t)\to 0$ uniformly in $y\in\RN$ when $t\to \infty$,
  one can find $t\geq T$ such that
  $\int_{B(0,R)}k_{\RN}(x_0,y,t)\leq \varepsilon/2$. Now, for that
  $t\geq T$, as $k_\RN(x_0,y,t)$ is integrable in $\RN$, we can find
  $\tilde{R}\geq R$ large enough such that
  $\int_{\RN\backslash B_{\tilde{R}}}k_{\RN}(x_0,y,t)\leq
  \varepsilon/2$ and the result is proved.
\end{proof}

\begin{proof}[Proof of Theorem \ref{thm:dowhatyouwant}]
First of all, without loss of generality we can assume that the
sequence is alternating, in the sense that $a_{2n-1}\leq a_{2n} \geq
a_{2n+1}$ for every $n\in \mathbb{N}$. If not, we add additional terms
into the sequence to make it alternating, prove the result in such a
case, and then get  the sequence of times  
for the corresponding subsequence. In the case of the alternating sequence, we
just need to prove that there exists a sequence of times such that
$S^\theta(t_{2n-1})u_0(x_0)\leq a_{2n-1} \Phi^\theta(x_0)$ and
$S^\theta(t_{2n})u_0(x_0)\geq a_{2n} \Phi^\theta(x_0)$, since by continuity
there exists $t\in [t_{2n-1},t_{2n}]$ such that $S^\theta(t)u_0(x_0)= 
a_{2n-1} \Phi^\theta(x_0)$ and  $s\in [t_{2n},t_{2n+1}]$ such
that $S^\theta(s)u_0(x_0)= a_{2n} \Phi^\theta(x_0)$. 
	
Let us construct $u_0$ by induction. By Theorem \ref{thm:dataLinfty_Dirichlet}, we
know that $\abs{S^\theta(t)u_0-\Phi^\theta(x)S_\RN(t)u_0}\leq f(t)$ where $f(t)$
is a monotonically decreasing function such that
$\lim_{t\to\infty}f(t)=0$. For convenience let us call
$g(t)=(\Phi^\theta(x_0))^{-1}f(t)$. The initial data $u_0$ is  the sum
of characteristic functions of annuli, defined as   
	\begin{equation}
		u_0(x):=\sum_{n=1}^\infty \chi_{B_{R_{2n}}\backslash B_{R_{2n-1}}}(x)
	\end{equation}
	where we will determine the radii $R_{n}$ by induction.
	 First, we choose $T>0$ large enough such that
         $\varepsilon=a_1-g(T)>0$ and $R_{0}>0$ such that $\hole\subset
         B(0,R_{0}) = B_{R_{0}}$. Hence, by Lemma \ref{lemma:predowhatyouwant}, there
         exists $t_1>T$ and $R_1>R_{0}$ such that 
	\begin{equation}
		\int_{\Omega\backslash B_{R_1}} k_\RN(x_0,y,t_1) \,
                dy  \leq a_1-g(t_1) .
	\end{equation}
	Now, choose  $T\geq t_1$ large enough such that $\varepsilon=1-a_2-g(T)>0$. Hence, by Lemma \ref{lemma:predowhatyouwant}, there exists $t_2>T$ and $R_2>R_{1}$ such that
\begin{equation}
	\int_{B_{R_2}\backslash B_{R_1}} k_\RN(x_0,y,t_2) \,  dy
        \geq a_2+g(t_2) . 
\end{equation}

Now we proceed by induction. Given $R_{2n}$ and $t_{n}$, we choose
$T\geq t_{2n}$ large enough such that
$\varepsilon=a_{2n+1}-g(T)>0$. Hence, by Lemma
\ref{lemma:predowhatyouwant}, there exists $t_{2n+1}>T$ and
$R_{2n+1}> R_{2n}$ such that
	\begin{equation}
		\label{eqn:dowhatyouwanteq1}
		\int_{B_{R_{2n}}}k_{\RN}(x_0,y,t_{2n+1}) \,  dy
                +\int_{\RN\backslash B_{R_{2n+1}}}
                k_\RN(x_0,y,t_{2n+1}) \,  dy \leq a_{2n+1}-g(t_{2n+1}) .
	\end{equation}
	Now, choose $T\geq t_{2n+1}$ large enough such that
        $\varepsilon=1-a_{2n+2}-g(T)>0$. Hence, by Lemma
        \ref{lemma:predowhatyouwant}, there exists $t_{2n+2}>T$ and
        $R_{2n+2}> R_{2n+1}$ such that 
	\begin{equation}
		\label{eqn:dowhatyouwanteq2}
	\int_{B_{R_{2n+2}}\backslash B_{R_{2n+1}}}
        k_\RN(x_0,y,t_{2n+2}) \,  dy \geq a_{2n+2}+g(t_{2n+2}) .
	\end{equation}

	Let us prove that $u_0$ satisfies the desired property for the sequence of times $\{t_n\}_{n\in\mathbb{N}}$.
\begin{equation}
	\label{eqn:dowhatyouwanteq3}
	\begin{aligned}
		S_\RN & (t_{2n+1}) u_0(x_0)  =\sum_{m=1}^\infty
                \int_{B_{R_{2m}}\backslash B_{R_{2m-1}}}
                k_\RN(x_0,y,t_{2n+1}) \, dy \\
		&\leq \int_{B_{R_{2n}}} k_{\RN}(x_0,y,t_{2n+1}) \,
                dy +\int_{\RN\backslash B_{R_{2n+1}}}
                k_\RN(x_0,y,t_{2n+1}) \,  dy
                \myleq{\eqref{eqn:dowhatyouwanteq1}}
                a_{2n+1}-g(t_{2n+1}) .
	\end{aligned}
\end{equation}
Therefore, by Theorem \ref{thm:dataLinfty_Dirichlet},
\begin{equation}
	S^\theta(t_{2n-1})u_0(x_0)\leq \Phi^\theta(x_0)
        S_\RN(t_{2n-1})u_0(x_0)+\Phi^\theta(x_0)g(t_{2n-1})\myleq{\eqref{eqn:dowhatyouwanteq3}}
        a_{2n-1} . 
\end{equation}
In the same way,
\begin{equation}
	\label{eqn:dowhatyouwanteq4}
	\begin{aligned}
		S_\RN(t_{2n})u_0(x_0) & =\sum_{m=1}^\infty \int_{B_{R_{2m}}\backslash B_{R_{2m-1}}}k_\RN(x_0,y,t_{2n}) \\
		& \geq \int_{B_{R_{2n}}\backslash B_{R_{2n-1}}}k_\RN(x_0,y,t_{2n})\mygeq{\eqref{eqn:dowhatyouwanteq2}} a_{2n}+g(t_{2n})
	\end{aligned}
\end{equation}
so then, by Theorem \ref{thm:dataLinfty_Dirichlet},
\begin{equation}
	S^\theta(t_{2n})u_0(x_0)\geq \Phi^\theta(x_0) S_\RN(t_{2n})u_0(x_0)+\Phi^\theta(x_0)g(t_{2n})\mygeq{\eqref{eqn:dowhatyouwanteq4}} a_{2n}
\end{equation}
as we wanted to prove.
\end{proof}

\section{Asymptotic behaviour for  initial data  in $L^p(\Omega), \ 1<p<\infty$}
\label{sec:initial_data_Lp}

In this section, we will study the heat equation with some homogeneous
$\theta$-boundary conditions in an exterior domain in the case in which the
initial data is in $L^p(\Omega)$ with $1<p<\infty$. This is the
simpler case because all the solutions decay to $0$ in
$L^{p}(\Omega)$.

\begin{proposition}
\label{prop:lpasym}

Let $S^\theta(t)$ be 
the solution semigroup of contractions of  the heat equation with homogeneous $\theta$-boundary
conditions on $\partial \Omega$.  

Then for any $1<p<\infty$ and  $u_0\in L^p(\Omega)$ 
\begin{equation}
\lim_{t\to\infty}\norm{S^\theta(t)u_0}_{L^p(\Omega)}=0 
\end{equation}	
and  for any $q$ such that $p< q \leq \infty$,
\begin{equation}
\lim_{t\to\infty}t^{\frac{N}{2}(\frac{1}{p}-\frac{1}{q})} \norm{S^\theta(t)u_0}_{L^q(\Omega)}=0.
\end{equation}
\end{proposition}
\begin{proof}
  The proof follows  by approximation. Given $\varepsilon>0$, by
  the density of $L^1(\Omega)\cap L^p(\Omega)$ in $L^p(\Omega)$ there
  exists a $u_0^\varepsilon\in L^1(\Omega)\cap L^p(\Omega)$ such that 
	\begin{equation}
		\norm{u_0-u_0^\varepsilon}_{L^p(\Omega)}\leq \varepsilon.
	\end{equation}
	Then, using Corollary \ref{cor:LpLq_estimates} we obtain, for
        $p\leq q \leq \infty$, 
	\begin{equation}
		\begin{aligned}
			t^{\frac{N}{2}(\frac{1}{p}-\frac{1}{q})}\norm{S^\theta(t)u_0}_{L^q(\Omega)}  & \leq t^{\frac{N}{2}(\frac{1}{p}-\frac{1}{q})}\left(\norm{S^\theta(t)u_0^\varepsilon}_{L^q(\Omega)}+\norm{S^\theta(t)(u_0-u_0^\varepsilon)}_{L^q(\Omega)}\right) \\ & \myleq{\eqref{eqn:LpLq_estimates_theta}}Ct^{\frac{N}{2}(\frac{1}{p}-1)}\norm{u_0^\varepsilon}_{L^1(\Omega)}+C\norm{u_0-u_0^\varepsilon}_{L^p(\Omega)} \\
			& \myleq{\phantom{\eqref{eqn:LpLq_estimates_theta}}} Ct^{\frac{N}{2}(\frac{1}{p}-1)}\norm{u_0^\varepsilon}_{L^1(\Omega)}+C\varepsilon \longrightarrow C\varepsilon
		\end{aligned}		
\end{equation}
when $t\to\infty$. Hence, as $\varepsilon$ was arbitrary, we have that
$\lim_{t\to\infty}t^{\frac{N}{2}(\frac{1}{p}-\frac{1}{q})}\norm{S^\theta(t)u_0}_{L^p(\Omega)}=0$.
 \end{proof}

\begin{remark}
  A formal explanation of  the decay to zero in $L^p(\Omega)$ for
  $1<p<\infty$ is the
  following.  Assume the initial data and hence the solution of the
  heat equation are nonnegative. If we multiply  in \eqref{eq:heat_theta}  by
  $u^{p-1}$ we obtain 
  \begin{equation}
    \frac{1}{p}\frac{d}{dt}\int_\Omega u^p =- C_p\int_\Omega u^{p-2}
    \abs{\nabla u}^2 + \frac{1}{p}\int_{\partial\Omega}\frac{\partial u^p}{\partial n},
  \end{equation}
where $C_{p}>0$. 
On the other hand, for $p=1$,
integrating the equation in $\Omega$, we obtain 
\begin{equation} \label{eq:loss_of_mass_L1}
	\frac{d}{dt}\int_\Omega u =\int_\Omega \Lap u 
        =\int_{\partial \Omega} \frac{\partial u}{\partial n} 
\end{equation}

So, as  we consider nonnegative solutions  $u\geq 0$ with, say,
Dirichlet boundary conditions,  we have
$\restr{u}{\partial \Omega}=0$ and then 
$\frac{\partial u}{\partial n} \leq 0$ on $\partial \Omega$ and then
$\D \int_\Omega u(x,t)dx$ and  $\D \int_\Omega u^{p}(x,t)dx$ decrease
in time. 

However, for the latter,   we can see a difference with respect to
\eqref{eq:loss_of_mass_L1} 
as  an additional decay term in $\Omega$ appears. Thus the equation is more
  dissipative in $L^p(\Omega)$ than in $L^1(\Omega)$.   
\end{remark}

It turns out that the decay in the
$L^p(\Omega)$ norm can be arbitrarily slow, as the following Lemma based on \cite{souplet1999geometry} shows.
\begin{lemma}
\label{lemma:souplet}
Let $g\in C([0,\infty))$ such that
$\lim_{t\to\infty}g(t)=0$. Then there exists an initial datum
$0\leq u_0\in L^p(\Omega)$ with $\|u_{0}\|_{L^p(\Omega)}=1$, and $T>0$ such that
	\begin{equation}
		\label{eqn:soupleteq1}
		\norm{S^\theta(t)u_0}_{L^p(\Omega)}\geq g(t) \qquad \forall t\geq T.
	\end{equation}

      \end{lemma} 
\begin{proof}
  The proof of this result can be found in \cite{souplet1999geometry}
  Proposition 3.3 iv) for  homogeneous Dirichlet
  boundary conditions. For other $\theta$-boundary conditions, 
  using \eqref{eqn:neugeqdir2}, 
\begin{equation}
\norm{S^\theta(t)u_0}_{L^p(\Omega)}\mygeq{\eqref{eqn:neugeqdir2}}
\norm{S^0(t)u_0}_{L^p(\Omega)}\mygeq{\cite{souplet1999geometry}}g(t)
\qquad \forall t\geq T. 
\end{equation}
\end{proof}

\section*{Acknowledgments}

The authors want to acknowledge fruitful discussions with F. Quirós, 
J. A. Cañizo and A. Garriz,  as well as the hospitality from  the Institute of Mathematics of the
University of Granada (IMAG). 

\appendix

\begin{appendices}

      \section{Schauder Estimates}
      Here we present some parabolic Schauder estimates, which allow
      us to estimate the derivatives of a solution of the heat
      equation just with the $L^\infty$ norm of the solutions. These
      are classical results which can be found, for example, in
      \cite{friedman2008partial} Chapter 3 Theorem 5.
            \begin{theorem}
	\label{thm:schauderest}
	Let $K\subset \RN$ a domain, $Q\defeq K\times [T_1,T_2]$ and
        $v\in L^\infty(Q)\cap C^{\infty}(Q)$ be a solution of the heat equation. Define,
        for any $(x,t)\in Q$ the parabolic distance
        $d_{(x,t)}=\inf\{(\abs{x-\bar{x}}^2+\abs{t-\bar{t}})^{1/2} \ :
        \ (\bar{x},\bar{t})\in \partial Q \backslash\{(x,T_2):x\in
        \Omega\} \}$. Then,
	\begin{equation}
          d_{(x,t)}\abs{Dv(x,t)}+d^2_{(x,t)}\abs{D^2v(x,t)}\leq C\norm{v}_{L^\infty(Q)} \qquad\forall (x,t)\in Q,
	\end{equation}
	where $C$ is independent of $v$, $x$ $t$, $K$, $T_1$ and
        $T_2$, $Dv$ represent any first order spatial derivative of
        $v$ and $D^2v$ any second order spatial derivative of $v$.
      \end{theorem}

  \section{Young's convolution inequality for Lorentz spaces}
\label{app:young}
  We state the definition of Lorentz spaces as well as some
  properties. We let the reference \cite{grafakos} for more
  information about these spaces.

\begin{definition}
Let $(X, \Sigma, \mu)$ be a measure space. For $0 < p < \infty$ and $0
< q \leq \infty$, the Lorentz space $L^{p,q}(X)$ is defined as the
space of measurable functions $f: X \rightarrow \mathbb{C}$ satisfying
the following condition: 
\[
\|f\|_{L^{p,q}(X)}^p := \begin{cases}
 \displaystyle \left(p\int_0^\infty t^{q - 1} \left(\mu(\{x : |f(x)| > t\})\right)^{q/p} dt\right)^{1/q}, & \text{if } q < \infty,\\
\displaystyle \sup_{t > 0} \left( t^p \mu\{x : |f(x)| > t\}\right), & \text{if } q = \infty,
\end{cases}
 \]
 where $\mu$ is the measure on $X$. For convention $L^{\infty,\infty}(X,\mu)=L^{\infty}(X,\mu)$.
 \end{definition}

 As a consequence of the definition
\begin{proposition}
\label{prop:lorentzprop}

We have the following properties
\begin{enumerate}
\item If $1\leq p \leq \infty$ we have $L^{p,p}(X)=L^p(X)$.
\item If $q\leq r$ then $L^{p,q}(X)\subset L^{p,r}(X)$ continuously.
\end{enumerate}
\end{proposition}

We state the following theorem for convolutions when we consider the
Lorentz spaces. Its proof can be found in \cite{convolution} Theorem
2.6.
\begin{theorem}
\label{thm:conv}
Let $f\in L^{p_1,q_1}(\RN)$ and $g\in L^{p_2,q_2}(\RN)$ such that
$\frac{1}{p_1}+\frac{1}{p_2}>1$. Then, the convolution $h=f\ast g$ is
well-defined and, for $p_3\geq 1$ such that 
\begin{equation}
\label{eqn:thmconveq1}
  \frac{1}{p_3}=\frac{1}{p_1}+\frac{1}{p_2}-1
\end{equation}
 and $q_3\geq 1$ such that
\begin{equation}
\label{eqn:thmconveq2}
\frac{1}{q_3}\leq \frac{1}{q_1}+\frac{1}{q_2}
\end{equation}
we have
\begin{equation}
\norm{h}_{L^{p_3,q_3}(\RN)}\leq 3p_3	\norm{f}_{L^{p_1,q_1}(\RN)}\norm{g}_{L^{p_2,q_2}(\RN)}
\end{equation}
\end{theorem}

This theorem, in particular, gives some estimates for convolution with
Gaussian functions.

\begin{proposition}
	\label{prop:convgauss}
	Let $f\in L^{p,\infty}(\RN)$ and $1\leq p \leq q \leq \infty$. If we consider the convolution
	\begin{equation}
		h(x,t)=\int_\RN \frac{e^{-\frac{\abs{x-y}^2}{4t}}}{(4\pi t)^{N/2}}f(y)dy
	\end{equation}
we have that there exists a constant $C(p,q)$ such that
\begin{equation}
	\label{eqn:convgausseq1}
	\norm{h(\cdot,t)}_{L^q(\RN)}\leq \frac{C(p,q)
          \norm{f}_{L^{p,\infty}(\RN)}}{t^{\frac{N}{2}(\frac{1}{p}-\frac{1}{q})}},
        \quad t>0 . 
\end{equation}
Note that, due to Proposition \ref{prop:lorentzprop}, the same result is obtained if $f\in L^p(\RN)$.
\end{proposition}
\begin{proof}
	Denote $G(x,t)=(4\pi t)^{-N/2}e^{-\abs{x}^2/4t}$. Then
        $\norm{G(t)}_{L^s(\RN)}\leq C(p)t^{-N/2(1-1/s)}$. Now, if
        $q\neq \infty$, we use Theorem \ref{thm:conv} with
        $g=G(\cdot,t)$, $(p_3,q_3)=(q,q)$, $(p_1,q_1)=(p,\infty)$ and
        $(p_2,q_2)=(r,r)$ with $1/r=1/q-1/p+1\geq 1/q$ (so
        \eqref{eqn:thmconveq1} and \eqref{eqn:thmconveq2} are
        satisfied) and obtain \eqref{eqn:convgausseq1}.

	If $q=\infty$ and $p\neq \infty$, we use
        \eqref{eqn:convgausseq1} to obtain
        $\norm{h(t/2)}_{L^p(\RN)}\leq
        C(p)\norm{f}_{L^{p,\infty}(\RN)}$. Then, we use
        $h(t)=G(t/2)\ast h(t/2)$ and standard Young's convolution
        inequality to obtain $$\norm{h(t)}_{L^\infty(\RN)}\leq
        \norm{G(t/2)}_{L^{p'}(\RN)}\norm{h(t/2)}_{L^p(\RN)}\leq
        \frac{C(p)\norm{f}_{L^{p,\infty}(\RN)}}{t^{\frac{N}{2p}}}.$$ 
	If $p=q=\infty$, it is straightforward also by standard Young's convolution inequality.
\end{proof}

    \end{appendices}

\newcommand{\etalchar}[1]{$^{#1}$}

\end{document}